\documentclass[12pt,oneside]{amsart}
\usepackage{geometry}              
\usepackage{graphicx}
\usepackage{amssymb}
\usepackage{epstopdf}
\usepackage{fourier}  

\DeclareMathAlphabet{\pazocal}{OMS}{zplm}{m}{n}
\usepackage{amsmath,amsthm,amscd,amssymb}
\usepackage{latexsym}
\usepackage[colorlinks,citecolor=red,pagebackref,hypertexnames=false]{hyperref}
\usepackage{mathtools}
\usepackage{float}

\newcommand{\N}{\ensuremath{\mathbb{N}}}
\newcommand{\Q}{\ensuremath{\mathbb{Q}}}
\newcommand{\R}{\ensuremath{\mathbb{R}}}

\newcommand{\al}{\alpha}

\newcommand{\ga}{\gamma}
\newcommand{\de}{\delta}

\newcommand{\Ga}{\Gamma}

\newcommand{\la}{\lambda}
\newcommand{\La}{\Lambda}

\newcommand{\Om}{\Omega}

\usepackage{hyperref}
\usepackage{color}

\numberwithin{equation}{section}
\theoremstyle{plain}
\newtheorem{theorem}{Theorem}[section]
\newtheorem{lemma}[theorem]{Lemma}
\newtheorem{corollary}[theorem]{Corollary}

\theoremstyle{definition}
\newtheorem{definition}[theorem]{Definition}

\theoremstyle{remark}
\newtheorem{fact}[theorem]{Fact}
\newtheorem{remark}[theorem]{Remark}

\theoremstyle{notation}

\newtheorem{acknowledgment}[theorem]{Acknowledgment}
\newtheorem{case[theorem]}{Case}

\makeatletter
\DeclareRobustCommand\widecheck[1]{{\mathpalette\@widecheck{#1}}}
\def\@widecheck#1#2{%
    \setbox\z@\hbox{\m@th$#1#2$}%
    \setbox\tw@\hbox{\m@th$#1%
       \widehat{%
          \vrule\@width\z@\@height\ht\z@
          \vrule\@height\z@\@width\wd\z@}$}%
    \dp\tw@-\ht\z@
    \@tempdima\ht\z@ \advance\@tempdima2\ht\tw@ \divide\@tempdima\thr@@
    \setbox\tw@\hbox{%
       \raise\@tempdima\hbox{\scalebox{1}[-1]{\lower\@tempdima\box
\tw@}}}%
    {\ooalign{\box\tw@ \cr \box\z@}}}
\makeatother

\definecolor{blue}{rgb}{0,0,1}

\definecolor{red}{rgb}{1,0,.2}

\usepackage{soul,xcolor}
\newcommand{\alert}{}

\newcommand{\ybox}{}
\author{K\'aroly Simon}

\address[K\'aroly Simon]{Budapest University of Technology and Economics, Department of Stochastics, Institute of Mathematics} \email{simonk@math.bme.hu}

\author{Krystal Taylor }

\address[Krystal Taylor]{Department of Mathematics, The Ohio State, Columbus, OH}
\email{taylor.2952@osu.edu}

\thanks{
}

\title{\parbox{14cm}{\centering{ Dimension and Measure of Sums of Planar sets and Curves  }}}

\begin{document}

\setstcolor{red}

\begin{abstract}
Considerable attention has been given to the study of the arithmetic sum
of two planar sets. We focus on understanding the measure and dimension of  $A+\Gamma:=\left\{a+v:a\in A, v\in \Gamma \right\}$ when  $A\subset \mathbb{R}^2$ and $\Gamma$ is a piecewise $\mathcal{C}^2$ curve.  
Assuming $\Ga$ has non-vanishing curvature,  we verify that
\begin{description}
  \item[(a)] if $\dim_{\rm H} A \leq 1$, then $\dim_{\rm H} (A+\Gamma)=\dim_{\rm H} A +1$;
  \item[(b)] if $\dim_{\rm H} A>1$, then $\pazocal{L}_2(A+\Gamma)>0$;
  \item[(c)] if $\dim_{\rm H} A=1$ and $\mathcal{H}^1(A) < \infty$, then $\pazocal{L}_2(A+\Gamma)=0$ if and only if $A$ is an irregular (purely unrectifiable) $1$-set.
\end{description}
In this article, we develop an approach using nonlinear projection theory which gives new proofs of (a) and (b) and the first proof of (c).  
Item (c) has a number of consequences:
 if a circle is thrown randomly on the plane, it will almost surely not intersect the 
 four corner Cantor set. Moreover, the pinned distance set of an irregular $1$-set has $1$-dimensional Lebesgue measure equal to zero at almost every pin $t\in \mathbb{R}^2$. 
\end{abstract}
\maketitle

\setcounter{tocdepth}{1}
\tableofcontents

\section{Introduction}
\subsection{The theme of the paper}
Given a  set $A\subset \mathbb{R}^2$, we study the set of points which are at a distance $1$ from at least one of the elements of $A$, where ``distance" refers to Euclidean distance or some other natural distance on the plane.
This set is $A+S^1$, where $S^1$ is the unit circle in the given distance.
More generally, we consider 
$$A+\Ga:= \{ a + g : a \in A, g \in \Ga\},$$
where $\Ga$ denotes a piecewise $\mathcal{C}^2$ curve, 
and we investigate
the following questions: under which conditions is the
\begin{description}
  \item[(Q1)] 2-dimensional Lebesgue measure of $A+\Gamma$ is positive?
  \item[(Q2)] Hausdorff dimension of $A+\Gamma$ the maximum of $1+\dim_{\rm H} A$ and $2$?
\end{description}

While these questions have been partially resolved using Fourier analytic methods, see Section \ref{S01}, the critical situation when $\dim_{\rm{H}}(A) =1$ is left open. We introduce a new proof technique that encompasses this case.  
%
%
The novelty in our approach is that we introduce a $1$-parameter family of Lipschitz maps $\left\{\Phi_\alpha\right\}_{\alpha\in I}$, where $I$ is an interval, $\Phi_\alpha:A\to \ell _\alpha$, and $\ell _\alpha$ is the vertical line at $x=\alpha$. This family $\left\{\Phi_\alpha\right\}$ is defined in such a way that
\begin{itemize}
  \item 
$A+\Ga$ can be expressed as the union $\cup_{a\in I} \Phi_\alpha(A)$.
  (In particular, $\cup_{a\in I} \Phi_\alpha(A)\subset A+\Gamma$ and the reverse containment holds with a proper extension of the mapping $\Phi_\al$). 
  \item $\left\{\Phi_\alpha\right\}_{\alpha\in I}$ satisfies the so-called transversality condition. That is, the intrinsic nature of the family $\left\{\Phi_\alpha\right\}_{\alpha\in I}$ is similar to that of the family of orthogonal projections on the plane. This makes it possible to invoke a number of results and methods from fractal geometry when we study questions \textbf{(Q1)} and \textbf{(Q2)}.
\end{itemize}

This work is continued in \cite{STint}, where we investigate the existence of an interior point in sum sets of the form $A+\Ga$.
It is also continued in \cite{BT21, CDT21, DT22}, where more quantitative estimates are obtained, and in \cite{McT} where pinned distance sets and paths are investigated in the more specific setting of Cantor sets.

The behavior of $A+\Gamma$ is conspicuously different when the piecewise-$\mathcal{C}^2$ curve $\Gamma$ has non-vanishing curvature versus the behavior when $\Gamma$ is a polygon. 
Our main results are presented in Sections \ref{results1} and \ref{results2}, respectively, based on these two cases.  
Relevant notation and definitions are collected in Section \ref{notation}.   
Before stating our main results, we consider some related results in the literature.

\subsection{History and  motivation}\label{S01}
Let $S(a,r)$ denote the circle in the plane with center $a$ and radius $r$, and identify the set of all such circles with $\mathcal{S}=\mathbb{R}^2\times (0,\infty)$.  Given a collection of center-radii pairs $E\subset \mathcal{S}$
 with dimension greater than $1$, it is reasonable to hypothesize that since a given circle has dimension $1$, then the union over circles in $E$ has dimension $2$.

In 1987, Marstrand  \cite{Mar87}  proved that if $A\subset\mathbb{R}^2$ is a set of positive Lebesgue measure, and if $r:\mathbb{R}^2 \rightarrow (0,\infty )$, then the set $$ \bigcup_{a \in A} S(a,r(a)) $$ cannot have zero Lebesgue measure.
This preliminary result holds in higher dimensions as a consequence of the Stein spherical
maximal theorem \cite{S76} established by Stein for $d \geq 3$ and by Bourgain \cite{Bo86}.
 Bourgain actually proved a stronger result that extends to sets with non-vanishing curvature in the case when the set of centers has positive measure.  
 See also joint work of the second listed author in \cite{IKST}, in which $L_p$ bounds are established for the Stein spherical maximal operator in the setting of compactly supported Borel measures; we note that this connection was an initial catalysts for the article at hand.

A considerable strengthening of Marstrand's result was proved by Wolff \cite[Corollary 3]{W00}, and a higher dimensional analog was provided by Oberlin \cite{O06}.  
In these works, it is proved that if 
$E\subset \mathbb{R}^d\times \mathbb{R}^+$ satisfies $\ybox{\dim_{\rm H} E>1}$, 
 then
    \begin{equation*}\label{O99}
   \pazocal{L}_d \left( \bigcup_{(a,\rho)\in E} S(a,\rho) \right)\ne 0, 
    \end{equation*}
where $\pazocal{L}_d$ denotes the $d$-dimensional Lebesgue measure.  

 Setting $\rho \equiv 1$, we obtain the following corollary:
  Let $A$ be a compact subset of  $\,\,\mathbb{R}^d$ for some $d\geq 2$.
  If $\dim_{\rm H} A>1$, then
   \begin{equation*}\label{O97}
    \alert{   \pazocal{L}_d \left(   A+S^1\right)>0}.
   \end{equation*}

 In Theorem \eqref{main_measdim} (a), we show that $\dim_{\rm H} (A)>1$ suffices to conclude that the \textit{Lebesgue measure} of $A+\Gamma $ is positive, where $\Gamma $ is an arbitrary $C^2$ curve with at least one point of non-zero curvature.  
 The necessity of the hypotheses on $\Gamma$ are considered in Theorem \eqref{084} (a').

 When $\dim_{\rm H} A \leq 1$, it is straightforward to verify that the dimension of $A+S^1$ is at most $\dim_{\rm H} A + 1$ (see \cite[Corollary 7.4]{Falc90}).
Wolff proved the reverse inequality \cite[Corollary 5.4]{W97}.
In particular, he showed that if 
  $A\subset \mathbb{R}^2$ is a Borel set with $\dim_{\rm H} A \leq 1$ then
  \begin{equation}\label{O95}
    \dim_{\rm H} (A+S^1) \geq 1+\dim_{\rm H} A.
  \end{equation}
Also see Oberlin \cite[(3.1)]{O08}, in which $S^1$ is replaced by a Salem set, as well as \cite{HT21}, where the Hausdorff and Fourier dimensions of more general sum and product sets are investigated.   

In Theorem \ref{main_measdim} (b), we show that $S^1$ can be replaced by a piecewise $C^2$ curve (see Definition \ref{085}) with at least one point of non-vanishing curvature.  
We provide both the known Fourier analytic proof of this fact, as well as a new proof using projection theory. 
In Theorem \ref{084} (b'), we prove that non-vanishing curvature at a point is a necessary condition.

\subsection{Notation}\label{notation}
Before we state our main results, we collect some notation and basic definitions.  
  \begin{enumerate}
    \item 
    Let $A \subset \mathbb{R}^d$, then $A^\circ$ denotes the interior, 
$\dim_{\rm H}(A)$ is the Hausdorff dimension, $\mathcal{H}^1(A)$ is the $1$-dimensional Hausdorff measure or ``length'' of $A$, and $\pazocal{L}_d(A)$ denotes the $d$-dimensional Lebesgue measure of $A$. 
\item 
Given $A,B\subset \mathbb{R}^2$, the Minkowski sum is given by
$$A+B=\{ a+b: a\in A, b\in B\}.$$
\item $S^1$ is the unit circle in the Euclidean distance, and $rS^1= \{ r \omega: \omega \in S^1\}$ for $r>0$. 
\item $\mathrm{proj}_\theta: \R^2 \rightarrow \R$ denotes the orthogonal projection, $\mathrm{proj}_\theta(x) = x\cdot e_\theta,$ where $e_\theta $ is the unit vector in direction $\theta$. 
\item $N_\theta$ denotes the angle $\theta$ rotation of the perimeter of $[0,1]^2$.
  \end{enumerate}

Cantor sets will play a role in examples and applications. 
\begin{definition}[Cantor set]\label{Cantor_defn}
A \textit{Cantor set} is a bounded totally disconnected perfect set (a closed set which is identical to its accumulation points).
For $\gamma\in(0,1)$, the \textit{symmetric Cantor set} $C_\gamma\subset [0,1]$  is defined as follows:
We iterate the same process that yields the usual middle-third Cantor set with the difference that we remove the middle-$1-2\gamma$ portion of every interval throughout the construction.  In this way, 
$$
  C_\gamma=\left\{(1-\gamma)\sum\limits_{k=1}^{\infty }
  a_k\gamma^{k-1}:a_k\in\left\{0,1\right\}
  \right\}.
$$
So, the middle $d$ Cantor set is $C_{\frac{1-d}{2}}$.
\end{definition}

\begin{definition}[Product Cantor set]\label{Cantor_defn}
For $\gamma\in(0,1)$, define the \textit{product Cantor set} $C(\gamma)\subset\mathbb{R}^2$ by
$$
      C(\gamma):=C_\gamma\times C_\gamma.
$$
Note, the Hausdorff dimension of $C(\gamma)$ is $(2\log 2)/(\log \gamma^{-1})$ (see \cite{Mat95} p. 60, 115). 
In particular, $C(1/4)$ is called the \emph{four corner Cantor set} and $\dim_{\rm H} C(1/4) = 1$.  
\end{definition}

Unless otherwise specified, the curves in this paper are assumed to be piecewise $\mathcal{C}^2$:
\begin{definition}[Piecewise $\mathcal{C}^2$ curve]\label{085}
  We say that a finite curve $\Gamma$ is \textit{piecewise} $\mathcal{C}^2$ if there is a parametrization $\underline{\gamma}(t):=(\gamma_1(t),\gamma_2(t))$, $t\in\left[0,T\right]$ and a partition
  $0=a_0<a_1< \cdots < a_n=T$ such that $\underline{\gamma}(t)$ is a  $\mathcal{C}^2$ curve
  on $t_i\in\left(a_{i-1},a_i\right)$
   for every $i=1, \dots ,n$ and $t\mapsto \underline{\gamma}(t)$, $t\in\left[0,L\right]$ is continuous.

   A piecewise   $\mathcal{C}^2$  curve is \textit{piecewise linear} if the curvature is zero for every $t\in(a_{i-1}-a_i)$, for every $i=1, \dots ,n$.
\end{definition}

\subsection*{Acknowledgements.}
This work came out of a collaboration initiated at ICERM at Brown University in Rhode Island and was then further supported through MRI at the Ohio State University. Simon is partially supported by the grant OTKA 123782,
by MTA-BME Stochastics Research Group.
Taylor is supported in part by the Simons Foundation Grant 523555.

\section{Main results when $\Gamma$ has non-vanishing curvature}\label{results1}

We collect two results in the next theorem, both of which follow from established Fourier analytic methods. 

\begin{theorem}[Measure and Dimension of $A+\Ga$]\label{main_measdim}
Let $\Gamma$ be a simple piecewise $\mathcal{C}^2$ curve with non-vanishing curvature 
and positive and finite length. 
Then for every compact set $A\subset \mathbb{R}^2$, it holds that 

  \begin{description}
    \item[(a)] if $\dim_{\rm H} A>1$, then  $A+\Gamma$ is a set of positive Lebesgue measure;
    \item[(b)] if $\dim_{\rm H} A \leq 1$, then $\dim_{\rm H} (A+\Gamma)=1+\dim_{\rm H} A$.
  \end{description}
\end{theorem}

We provide a new proof of Theorem \ref{main_measdim}
using the transversality method from fractal geometry.  The benefit of this technique is that it can be used to handle the critical case when ${\rm dim}_H(A)= 1$, as is done in Theorem  \ref{main_critical}.  For the sake of comparison, and as they are not very long, we also include Fourier analytic proofs for each part of Theorem \ref{main_measdim}. 
The proof based on nonlinear projection theory is given in Sections \ref{r88}- \ref{main_measdim_sec}.  Section \ref{checking_section} is dedicated to checking the transversality condition.  The Fourier analytic proofs are given in Section \ref{section_fourier}.

If $A \subset \mathbb{R}^2$ with $\dim_{\rm H} A=1$ then
Theorem \ref{main_measdim} says that $\dim_{\rm H} (A+S^1)=2$. A natural question, which we now turn to investigating, is under which condition can we decide if $\pazocal{L}_2(A+S^1)>0$? 
The answer depends on whether $A$ is regular or irregular.  

\begin{definition}[Rectifiable and purely unrectifiable sets]
We say that a set $A \subset \mathbb{R}^2$ is \textit{regular} or \textit{rectifiable} if $\mathcal{H}^1$-almost all of its points can be covered by a countable union of rectifiable curves. On the other hand, $A$ is \textit{irregular} or \textit{purely unrectifiable} if $\mathcal{H}^1(\gamma\cap A)=0$ for every rectifiable curve $\gamma$. 
\end{definition}
The four corner Cantor set, for instance, is a classic example on an irregular set with positive and finite length.

\begin{theorem}[Characterizing measure of $A+\Ga$ in the critical setting when $\dim_{\rm H}(A) = 1$]\label{main_critical}
Let $\Gamma \subset \mathbb{R}^2$ be a simple piecewise $\mathcal{C}^2$ curve with 
positive and finite length such that the curvature of $\,\Ga$ does not vanish except at a set of points having zero $\mathcal{H}^1$-measure.
Let $A\subset \mathbb{R}^2$ denote a compact set with $\dim_{\rm H} A=1$. Further we assume that $\mathcal{H}^1|_A$ is $\sigma$-finite.
Then
$\pazocal{L}_2(A+\Gamma)=0$ if and only if for every
rectifiable curve $\gamma$, we have $\mathcal{H}^1(\gamma\cap A)=0$.
\end{theorem}

The proof of Theorem \ref{main_critical} is given in Section \ref{g88}.
An immediate consequence is the following.
\begin{corollary}[Measure of $A+\Ga$ in the critical setting when $\dim_{\rm H}(A) = 1$]\label{g97}
  Let $A \subset \mathbb{R}^2$ be a compact $1$-set ($\mathcal{H}^1$-measurable and $0<\mathcal{H}^1(A)<\infty $). Then
  \begin{description}
    \item[(a)] $\pazocal{L}_2(A+S^1)=0$ whenever $A$ is irregular;
        \item[(b)] $\pazocal{L}_2(A+S^1)>0$ whenever $A$ is regular.
  \end{description}
\end{corollary}

\begin{remark}
  In the special case when $A$ is the attractor of a self-similar iterated function system which contains a map with irrational rotation in its linear part, part (a) follows from a result of A. Farkas \cite{Far}.
\end{remark}

Another immediate corollary of Theorem \ref{main_critical} is:

\begin{corollary}[Probabilistic consequence for Buffon circle problem]\label{probability}
Let $A \subset \mathbb{R}^2$  be a compact irregular $1$-set.
   Assume that we throw randomly  a unit circle on the plane such that the distribution of the center is absolutely continuous w.r.t. the two dimensional Lebesgue measure $\pazocal{L}_2$. Then this random circle does not intersect $A$ almost surely.
\end{corollary}
\begin{proof}[Proof of the Corollary]
  The random circle $S(x,1)$ intersects $A$ if and only if $x\in A+S^1$. But this happens with zero probability from Theorem \ref{main_measdim}.
\end{proof}

We note that a deeper investigation of the result of Corollary \ref{probability} is followed up on in \cite{BT21, CDT21}, where we obtain upper and lower bounds on the rate of decay as $n\rightarrow \infty$ of the probability that a unit circle intersects the $n$-th generation in the construction of the four corner Cantor set, $C(1/4)$.  Further, in \cite{BT21, DT22}, we obtain similar quantitative information for more general irregular $1$-sets.  

It is interesting to note that, while $F+S^1$ has positive measure when $F$ is a regular $1$-set, $F+S^1$ need not have an interior point.  
\begin{remark}\label{g91} There exists a regular $1$-set $F \subset \mathbb{R}^2$ such that $\left(F+S^1\right)^{\circ}=\emptyset$.
\end{remark}

\begin{proof}[Proof of Remark \ref{g91}]
 Let $F_0:=[0,1]\setminus\mathbb{Q}$, $\gamma_{1}(x):=\sqrt{1-x^2}$, and $\gamma_2(x):=-\sqrt{1-x^2}$.
 Setting $F = F_0\times \{0\}$, 
 it is immediate that, for $i=1,2$, 
 $$
\left( F + S^1 \right)\cap \left(\mathbb{Q}\times \gamma_i(\mathbb{\widetilde{\Q}})\right)= \emptyset,
$$
where $\widetilde{Q}= \Q\cap [-1,1]$. 
On the other hand, $\bigcup_{i=1,2}\left(\mathbb{Q}\times \gamma_i(\mathbb{Q})\right) $ is dense in the strip $\R \times [-1,1]$. 
 \end{proof}

\subsection{Pinned distance sets}
In the following corollary of Theorem \ref{main_measdim}, we consider the pinned distance sets
\begin{equation}\label{pinned_defn}  \Delta_{b}(A)=
  \left\{ \|a-b\|_2 : a\in A
  \right\}\end{equation} of a set $A \subset \mathbb{R}^2$ pinned at $b\in\mathbb{R}^2$.

The celebrated Falconer's distance conjecture is a famous problem in geometric measure theory, which concerns the relationship between the dimension of a subset $E$ of $\R^d$ and its distance set defined by 
$\Delta(A) := \{\|x-y\|: x,y\in A.\}$.  
In the plane, the conjecture states that if 
$\dim_{\rm{H}}(A)>1$, then $\pazocal{L}_1(\Delta(A))>0$.  
Recent exciting progress on the study of pinned distance sets includes work by 
Guth, Iosevich, Ou, and Wang \cite{GIOW}, T. Orponen, \cite{O12}, and P. Shmerkin, \cite{S17}.
 In \cite{GIOW}, it is shown that if $A\subset \R^2$ with $\dim_{\rm{H}}>\frac{5}{4}$, then there exists a $b\in A$ so that $\Delta_b(A) $ has positive Lebesgue measure.  
Orponen proves that if $A\subset\mathbb{R}^2$ is a self-similar set of $\dim_{\rm{H}}(A)>1$ with at least one irrational rotation in its construction, then $\dim_{\rm{H}}(\Delta_x(A))=1$ for some $x$.
Shmerkin proves that if $A\subset\mathbb{R}^2$ so that $\dim_{\rm{H}}(A)=\dim_{p}(A)> 1$, where $\dim_{p}$ denotes the packing dimension of $A$, then $\dim_{\rm{H}}(\Delta_x(A))=1$ for all $x$ outside of a set of dimension at most $1$.  
For more results on pinned distance sets, see for instance Peres and Schlag's \cite{PeSc00}, work by the second listed author in \cite{IMT12, ITU}, as well as our companion paper \cite[Section 2.1.2]{STint}. For analogous work on finite point configurations, see \cite{GIT19, GIT21} and the references therein. 

The following is a result toward understanding the Lebesgue measure at the critical dimension in the plane $\dim_{\rm{H}}(A) =1.$

\begin{corollary}[The pinned distance set of an irregular $1$-set has zero measure a.a.]\label{y99}
  Let $A \subset \mathbb{R}^2$ be a compact irregular $1$-set. Then
  \begin{equation}\label{Y98}
     \pazocal{L}_1\left(\Delta_{t}(A)\right)=0\qquad \mbox{for $\pazocal{L}_2$- a.a. $t\in\mathbb{R}^2$.}
  \end{equation}
 \end{corollary}
\begin{proof}
  This is immediate from Theorem \ref{main_critical}  and from the Fubini theorem. More precisely,
  fix an arbitrary $r>0$, and  observe that
  $$A+rS^1= \left\{y\in\mathbb{R}^2: A\cap S(y,r)\ne \emptyset\right\}.$$
 If we apply Theorem \ref{main_critical} for $(1/r) \cdot A$ instead of $A$, then we obtain that
 \begin{equation}\label{y97}
 \pazocal{L}_2\left\{y\in\mathbb{R}^2: A\cap S(y,r)\ne \emptyset\right\}=0, \qquad \forall r>0.
 \end{equation}

  Let
  $$
  \Theta:=\left\{(y,r)\in\mathbb{R}^2\times(0,\infty ):
  A\cap S(y,r)\ne \emptyset
  \right\}
  $$
 Then $\Theta^c \subset \mathbb{R}^3$ is open (since we assumed that $A$ is compact) so, $\Theta$ is closed.  By
 \eqref{y97} and the Fubini Theorem, we know that $\pazocal{L}_3\left(\Theta\right)=0$. Hence, also by the Fubini Theorem, we obtain that
 \begin{equation}\label{y96}
   \pazocal{L}_1 \left\{r>0:
   A\cap S(y,r)\ne \emptyset\right\} =0,\mbox{ for $\pazocal{L}_2$-a.a. $y\in\mathbb{R}^2$.}
 \end{equation}
  Clearly, $A\cap S(y,r)\ne \emptyset$ holds if and only if $r\in \Delta_{y}(A)$.
\end{proof}

\begin{remark}\label{y95}
  We remark that in our companion paper ,\cite{STint}, we investigate the interior of the pinned distance set of Cantor-sets of the form  $C(\gamma)$. Summarizing
  these results, Corollary \ref{y99}, and Peres-Schlag's result \cite[Corollary 8.4]{PeSc00}, we have:
  \begin{enumerate}
    \item If $\gamma=1/4$ then $\pazocal{L}_1\left(\Delta_{t}(C(1/4))\right)=0$ for $\pazocal{L}_2$- a.a. $t\in\mathbb{R}^2$ (Corollary \ref{y99}.)
    \item If $\gamma>\frac{1}{4}$ then  $\pazocal{L}_1\left(\Delta_{t}(C(\gamma))\right)>0$
    holds for almost all $t$ in the stronger sense that
     for every straight line $\ell  \subset \mathbb{R}^2$,
     for $\mathcal{H}^1$-almost all $t\in\ell $, $\pazocal{L}_1\left(\Delta_{t}(C(\gamma))\right)>0$.
     (\cite[Corollary 8.4]{PeSc00})
    \item If $\gamma \geq \frac{1}{3}$ then $\left(\Delta_{t}(C(\gamma))\right)^{\circ}\ne \emptyset $ for all $t\in \R^2$.
    (\cite[Corollary 2.12]{STint} and \cite[Theorem 2.15]{STint} ).
  \end{enumerate}
\end{remark}
It remains an open question which is the smallest $\gamma$ such that $\left(\Delta_{t}(C(\gamma))\right)^{\circ}\ne \emptyset $ for most of the $t$?

\subsection{Applications to product Cantor sets.}
Let $C(\ga)$ denote a product Cantor set as in  Definition \ref{Cantor_defn}.  
 \begin{remark}        
 What we know about the set $C(\gamma)+S^1$  is as follows:
\begin{enumerate}
  \item if $\gamma < \frac{1}{4}$ then $\dim_{\rm H} \left(C(\gamma)+S^1\right)=1+\frac{\log 4}{\log \gamma^{-1}}$. (See e.g. Theorem \ref{main_measdim}).
  \item if $\gamma=\frac{1}{4}$ then $\dim_{\rm H} \left(C(\frac{1}{4})+S^1\right)=2$ but $\pazocal{L}_2\left(C(\gamma)+S^1\right)=0$.
      (See Corollary \ref{g97}).
  \item if $\gamma>\frac{1}{4}$ then $\pazocal{L}_2\left(C(\gamma)+S^1\right)>0$. (See Theorem \ref{main_measdim}).
  \item if $\gamma \geq \frac{1}{3}$ then
  $\left(C(\gamma)+S^1\right)^{\circ}\ne \emptyset $ (see our companion paper
  \cite[Theorem 2.7]{STint}
   ).
   \item Further, $ \pazocal{L}_1\left(\Delta_{t}(C(\frac{1}{4}) )\right)=0 \,  \mbox{ for $\pazocal{L}_2$- a.a. $t\in\mathbb{R}^2$.}$ (See Corollary \ref{y99}).
\end{enumerate}
We do not know if there are  $\gamma\in\left(\frac{1}{4},\frac{1}{3}\right)$ with
$\left(C(\gamma)+S^1\right)^{\circ}\ne \emptyset $.
Also, regarding (5), it would be very interesting to obtain an upper bound on the Hausdorff dimension of the set of exceptional  $t$.  
\end{remark}

  \section{Main results when $\Gamma$ is a polygon}\label{results2}
In this section we always assume that $\Gamma$ is a piecewise linear curve.
First, we establish some connections between the size of $A+\Gamma$ and
$\mathrm{proj}_{\theta}(A)$,
the orthogonal projection in the plane to the angle$-\theta$ line.
Then we use Theorems of Kenyon \cite{Ken97} and Hochman \cite{Hoch14} on the projection properties of $A$ in various special cases to obtain the size of $A+\Gamma$.
We begin with a simple Lemma.  

\begin{lemma}[Measure and dimension of $A+\Ga$ when $\Ga$ is a union of straight lines]\label{g94}
Let $A \subset \mathbb{R}^2$ be Lebesgue measurable.
Let  $\Gamma=\bigcup_{i=1}^{n}I_i$, where $I_i$ are straight line segments and we write
$\alpha_i$ for the angle (with the positive part of the $x$-axis) of $I_i$. Then we have:
\begin{description}
  \item[(a)] if $\,\exists i$ $\pazocal{L}_1\left(\mathrm{proj}_{\alpha_i^{\perp}}(A)\right)>0$, then $\pazocal{L}_2\left(A+\Gamma\right)>0$;
  \item[(b)] if $\,\forall  i$ $\pazocal{L}_1\left(\mathrm{proj}_{\alpha_i^{\perp}}(A)\right)=0$, then $\pazocal{L}_2\left(A+\Gamma\right)=0$;
 \item[(c)] if $\,\exists i$ $\left(\mathrm{proj}_{\alpha_i^{\perp}}(A)\right)^{\circ}\ne\emptyset$, then $\left(A+\Gamma\right)^{\circ}\ne\emptyset$;
 \item[(d)] $\dim_{\rm H} (A+\Gamma) =1+\max\limits_{i=1, \dots n}\left(\dim_{\rm H} \left(\mathrm{proj}_{\alpha_i^{\perp}}(A)\right)\right)$,
\end{description}
where $\mathrm{proj}_{\alpha_i^{\perp}}(A) = \{ a \cdot e_i: a\in A\} \subset \R$ and $e_i$ denotes the unit vector in direction 
$\alpha_i+\frac{\pi}{2}$.
\end{lemma}
The proofs of statements (a)-(d) of Lemma \ref{g94} follow from fairly straightforward covering arguments; we omit the details.

Using Lemma \ref{g94}, we can prove the following:

\begin{theorem}[Non-vanishing curvature is a necessary condition of Theorem \ref{main_measdim}]\label{084}
  If $\ \Gamma$ is piecewise linear then neither \textbf{(a)} nor \textbf{(b)} of Theorem \ref{main_measdim} hold. That is
  \begin{description}
    \item[(a')] $\exists A\subset\mathbb{R}^2$ with
   $\dim_{\rm H} A>1$ and $\pazocal{L}_2\left(A+\Gamma\right) =0$
    \item[(b')]  $\exists A\subset \mathbb{R}^2$
    with
   $\dim_{\rm H} A<1$ and $\dim_{\rm H} (A+\Gamma)<1+\dim_{\rm H} A$.
\item[(c')] $\exists A,B\subset \mathbb{R}^2$ with
$\dim_{\rm H} A=\dim_{\rm H} B=1$, $\pazocal{L}_2(A+S^1)=0$, and $(B+S^1)^{\circ}\ne\emptyset$.
  \end{description}
  Actually in both cases $(a')$ and $(b')$, we can choose $A$ to be self-similar with strong separation condition. That is we can find a finite list of contracting similarities $\left\{S_1, \dots ,S_m\right\}$, $S_i:\mathbb{R}^2\to\mathbb{R}^2$ such that
  $A=\bigcup\limits_{i=1}^mS_i(A)$ and $S_i(A)\cap S_j(A) =\emptyset $ for $i,j\in\left\{1, \dots ,m\right\}$, $i\ne j$.
\end{theorem}
The proof of Theorem \ref{084} relies on an iterated function system, as well as Hutchinson's Theorem, and is given in Section \ref{P98}.

\subsection{The sum of the four corner Cantor set and a rotated square  } 
Here we always assume that  $\Gamma=N_\theta$, the angle $\theta$ rotation of the perimeter of $[0,1]^2$. 
When $\Gamma=N_{\theta}$ we can completely describe the size of $C(1/4)+ N_{\theta}$, for every $\theta$.
We follow Kenyon \cite{Ken97} (see also \cite[Section 10.3]{Mat15}.
 \begin{definition}[Small and Big rational angles]\label{g81}
Let $m$ be an integer and let $j_0:=\max\left\{j: 4^j \mbox{ divides } m\right\}$. Let
$m^*\in\left\{0,1,2,3\right\}$  be defined by
$$
m^*:=\frac{m}{4^{j_0}} \bmod{4}.
$$
That is $6^*=2$ since $6$ is not multiple of any power of $4$  but $112^*=3$ since $112=7 \cdot 4^2$ and $7 \bmod{4}$ is $3$.
We say that an angle $\theta\in[0,\pi)$ is \textit{rational} if $\tan\theta\in\mathbb{Q}$.
We partition the rational angles as follows:
$$
\mathrm{Small}:=\left\{\theta\in[0,\pi):
\tan\theta=\frac{p}{q},  \ p^*, q^* \mbox{ are both odd numbers},
\right\},
$$
$$
\mathrm{Big}:=\left\{\theta\in[0,\pi):
\tan\theta=\frac{p}{q}, \, \mbox{ either $p^*$ or $q^*$ is an even number},
\right\},
$$
where we assumed in both of the formulas  above that the  greatest common divisor of $p$ and $q$ is $1$.
\end{definition}
Kenyon  gave a full characterization (see \cite[Section 10.3]{Mat15}) of the $\mathrm{proj}_\theta$-projection of $C(1/4)$ for all $\theta\in[0,\pi)$:
\begin{theorem}[Kenyon]\label{g66} If:
\begin{itemize} 
    \item[(a)]  $\theta\in \mathrm{Small}$, then
    $\dim_{\rm H} \left(\mathrm{proj}_\theta C(1/4)\right)<1$.
    \item[(b)] 
    $\tan\theta\not\in \mathbb{Q}$, then
    $\pazocal{L}_1(\mathrm{proj_\theta}C(1/4))=0$.
    \item[(c)] $\theta\in \mathrm{Big}$, then
    $\left(\mathrm{proj}_\theta C(1/4)\right)^{\circ}\ne\emptyset$.
  \end{itemize}
\end{theorem}

Combining Lemma \ref{g94}, Theorem \ref{g66}, and the work of Hochman in \cite{Hoch14}, we have
\begin{theorem}[Measure, dimension, and interior of the sum of the four corner Cantor set and a rotated square]\label{g80} If:
\begin{itemize} 
    \item[(a)] 
    if $\theta\in \mathrm{Small}$, then $\dim_{\rm H} (C(1/4)+N_\theta)<2$.
    \item[(b)] 
    $\tan\theta\not\in\mathbb{Q}$, then
    $\pazocal{L}_2(C(1/4)+N_\theta)=0$ but $\dim_{\rm H} (C(1/4)+N_\theta)=2$.
    \item[(c)] if $\theta\in \mathrm{Big}$, then $(C(1/4)+N_\theta)^{\circ}\ne\emptyset$.
  \end{itemize}
\end{theorem}

\begin{proof}[Proof of Theorem \ref{g80}]
The fact that
 $\dim_{\rm H} (C(1/4)+N_\theta)=2$ can be seen as follows: inspecting the proof of Hochman's main result in \cite{Hoch14} implies
 that $\dim_{\rm H} \left(  \mathrm{proj}_\theta(C(1/4)) \right)=1$ for all irrational $\theta$.
Then we apply part (d) of Lemma \ref{g94}. All other assertions of Theorem \ref{g80} are  immediate combinations of Kenyon's Theorem above and Lemma \ref{g94}.
\end{proof}

Actually, in this case Palis-Takens conjecture (see \cite{PT93} or \cite{MG}) holds: $C(1/4)+N_\theta$ is either big in a sense that it contains some interior points or small in the sense that is has zero Lebesgue measure but it never happens that $C(1/4)+N_\theta$
has positive Lebesgue measure with empty interior.

\begin{remark}\label{000}
We have considered the cases when the curvature of $\Gamma$ either vanishes everywhere or vanishes nowhere. Now we address the in-between situation.
  We remark that one can construct a $\mathcal{C}^2$ curve $\Gamma$ containing no straight line segment and an irregular $1$-set $A$ such that $\pazocal{L}_2(A+\Gamma)>0$. 
Namely, for every $\theta$
let $A(\theta)$ be the  angle $\theta$ rotation of the four corner Cantor set, $C(1/4)$. We write  $A_2(\theta)$ for the orthogonal projection
of $A(\theta)$ to the $y$-axis.  It follows from part (c) of Theorem \ref{g66} that we can choose such a $\theta$ that
$\pazocal{L}_1(A_2(\theta))>0$.
We define $\Gamma$  as the graph of a  $\mathcal{C}^2$  function
  $\phi:[0,1]\to \mathbb{R}$ for which the set
  $$
  U:=\left\{(t,0),t\in[0,1],\phi''(t)=0,\phi(t)=0\right\}
  $$
  is a Cantor set of positive one-dimensional Lebesgue measure on the $x$-axis;
  note, this can be done in such a way that $U$ does not contain an interval.  
  Then $U \subset \Gamma$ and by Fubini Theorem 
  $\pazocal{L}_2(A+U)>0$. Consequently we have
$\pazocal{L}_2(A+\Gamma)>0$.
\end{remark}

\section{Proof of Theorem \ref{main_measdim}}\label{N30}
In this section, we present two proofs of Theorem \ref{main_measdim}.  
First, we present an elegant proof using established Fourier analytic methods in Section \ref{section_fourier}.
Second, we present a lengthier proof based on the transversality method.  The merit of this second lengthier proof is that it sets the foundation to prove Theorem \ref{main_critical}.    
The proof based on transversality will proceed in four sections:
In Section \ref{r88} we set up the necessary notation, consider the pairs of $\la$ and $a$ for which $\Phi_\la(a)$ is defined, and compare the measure of the sum set $A+\Ga$ to the measure of the images $\Phi_\la(A)$. 
The main tool, the transversality method, is introduced in Section \ref{KT50} along with some background. 
In Section \ref{main_measdim_sec}, we use the transversality method to complete the proof of Theorem \ref{main_measdim}.
Finally, the main technical work is done in Section \ref{checking_section}, where we verify that our family of maps satisfies the transversality condition.  

\subsection{Proof of Theorem \ref{main_measdim} with Fourier analysis and energy integrals}\label{section_fourier}

Theorem \ref{main_measdim} can be proved using well known properties of the Fourier transform.  
The proof in this section has advantages and disadvantages over the proof presented below in sections \ref{r88}- \ref{checking_section}.  
The proof here 
 is shorter, which is due to the fact that it relies on previously established theory and requires less set up. 
 Moreover, the proof in this section is very adaptable to higher dimensions, and analogous statements are available for $\mathbb{R}^d$, $d\geq 2$. 
Nevertheless, Theorem \ref{main_critical} seems to fall out of the scope of Fourier analytic methods, at least out of those used here, while our proof based on transversality can be used to handle the critical dimension when $\dim_{\rm H} (A) = 1$.  

\begin{proof}[Proof of Theorem \ref{main_measdim} part (a)]
 Assume that $A\subset [0,1]^2$ with $\dim_{\rm H}(A) >1$.  
Recall from \cite[Theorem 2.8]{Mat15}, for instance, the energy characterization of the definition of the Hausdorff dimension: 
 $$\dim_{\rm H} (A)=   \sup\left\{s:\exists \mu\in\mathcal{M}(A),  I_s(\mu)<\infty   \right\},$$
  where $  I_s(\mu) = \iint |x-y|^{-s} d\mu(x)d\mu(y)$ and $\mathcal{M}(A)$ is the set of non-zero, finite, Borel measures with compact support so that $\mathrm{spt}(\mu)\subset A$. 
  It follows that 
  for each $s<\dim_{\rm H}(A)$, there exists a measure $\mu \in \mathcal{M}(A)$ so that  $I_s(\mu)< \infty$.  
Further, recall that for $0<s<2$
  \begin{equation}\label{O91}
  \ybox{I_s(\mu)=\gamma(2,s) \cdot \int |\widehat{\mu}(\xi)|^2 \cdot |\xi|^{s-2}d\xi}
\end{equation}
for some positive constant $\ga(2,s)$ \cite[Theorem 3.10]{Mat15}. 
Combining these observations, 
we see that for each $s<\dim_{\rm H}(A)$, there exists a measure $\mu \in \mathcal{M}(A)$ so that
\begin{equation}\label{energyA}
\alert{   \int |\widehat{\mu}(\xi)|^2 \cdot |\xi|^{s-2}d\xi     <\infty} \ \mbox{ and }\mathrm{spt}(\mu)\subset A.
\end{equation}

 If $\Gamma'$ is a $C^2$ subcurve of $\Gamma$ with non-vanishing curvature, then there exists a measure $\sigma_{\Ga}$ supported in $\Gamma'$ with
$|\widehat{\sigma_{\Ga}}(\xi)| \lesssim (1+|\xi|)^{-\frac{1}{2}}$. 
(see, for instance, \cite[p.347 - 348]{Stein93} or \cite{W04}). 

 Our aim is to verify that $\pazocal{L}_2(A+\Gamma)>0$. 
Recall that for $\nu\in \mathcal{M}(\mathbb{R}^2)$,
if $\widehat{\nu}\in L^2(\mathbb{R}^2)$, then
$\nu$ is absolute continuous w.r.t. the Lebesgue measure with $\pazocal{L}_2$ density \cite[Theorem 3.3]{Mat15}. 
 Applying this fact with $\nu= \mathrm{spt}(\mu*\sigma_\Gamma)$
 and observing that $\mathrm{spt}(\mu*\sigma_\Gamma)\subset A+S^1$,  it suffices to show  that
\begin{equation}\label{O88}
 \int|\widehat{\mu*\sigma_\Gamma}(\xi)|^2d\xi<\infty .
\end{equation}

Combining the basic facts above, we have
\begin{equation}\label{O89}
  \int_{|\xi|>1}|\widehat{\mu*\sigma_\Gamma}(\xi)|^2 d\xi
   \lesssim
   \int_{|\xi|>1}|\widehat{\mu}(\xi)|^2 \cdot |\xi|^{-1}d\xi\sim I_1(\mu)<\infty.
\end{equation}
The integral over $\{|\xi|\le 1\}$ is trivially bounded by the finiteness of measures $\mu$ and $\sigma_{\Ga}$. 
\end{proof}

\begin{proof}[Proof of Theorem \ref{main_measdim} part (b)]
A simple energy integral proof recovers the proof of part (b); the details are given in D. Oberlin \cite{O08}, and we include them here to ease of reading.  We use the notation introduced above in part (a).
 Assume that $A\subset \mathbb{R}^2$ with $\dim_{\rm H}(A) \le1$.  As in part (a), for all $s<\dim_{\rm H}(A)$, there exists a measure $\mu$ on $A$ so that
\begin{equation}
\alert{   \int |\widehat{\mu}(\xi)|^2 \cdot |\xi|^{s-2}d\xi     <\infty} \ \mbox{ and }\mathrm{spt}(\mu)\subset A.
\end{equation}
It is straightforward to verify, using the properties provided in part (a), that $I_{\alpha}(\mu*\sigma_\Gamma)$ is finite whenever $\alpha<1 + \dim_{\rm H}(A)$.  It follows then, by the characterization of Hausdorff dimension given in terms of energy integrals (see part (a)), that $\dim_{\rm H}(A+\Gamma) \geq 1+ \dim_{\rm H}(A). $ Equality follows by \eqref{Lip} above. 
\end{proof}

 \subsection{Preparation for the transversality method}\label{r88}
We now move to the second proof of Theorem \ref{main_measdim}.  
The main goal of this section is establishing a simple set up so that the images $\Phi_\la(a)$ are defined as $\la$ ranges in a fixed parameter interval and $a$ varies overs $A$; the key equation established in this section is \eqref{Fubini_eq}. 
First, we make some simplifying assumptions on the curve,   
$\Gamma$, a piecewise $\mathcal{C}^2$ curve with non-vanishing curvature and positive length. 
\begin{definition}[Good curves are strictly monotone $\mathcal{C}^2$ graphs with bounded growth]\label{good_curve}
Suppose that $\ga$ is a real-valued $ \mathcal{C}^2$ function on a non-trivial compact interval $I\subset \R$ so that  $|\ga'| \le 1$, 
$\gamma'$ is strictly monotonic on $I$,
 and $\ga''\neq 0$ on $I$.  
We call a curve $\Ga$ in the plane \textit{good} if it can be expressed as the graph 
$$\Ga= \{(t, \ga(t)) : t\in I\},$$ 
of such a function $\ga$ over $I$.  
\end{definition}

Note that any simple $\mathcal{C}^2$ curve $\Ga$ of finite length can be decomposed into  
 finitely many subcurves $\{\Gamma_i\}$, possibly with overlapping end-points, so that each subcurve is a good curve or a rotation of a good curve.  
As it is enough to establish the dimension or measure of $A+\Ga_i$ for each such subcurve, by applying a rotation of the axes, we may simply assume that $\Ga$ is itself a good curve.  
 Further, by shifting $\Ga$, we may assume that 
 $$I=[0,L]$$
 and $\Ga \subset I^2 = [0,L]^2.$ 
 By applying a reflection, we can further assume that
 \begin{equation}\label{concave}
 \gamma''<\La< 0
 \end{equation}
 on $I$ for some $\La<0$
  so that $\Ga$ is concave down.

 Set $\Om= \left[0, \frac{L}{2}\right]^2$ and let 
 $A\subset \Om$ be a Borel subset of $\R^2$.
We first establish the theorem for such sets $A$.   
The theorem for general Borel sets $A$ follows from 
decomposing the plane into squares $\{S_i\}_i$ of side length $\frac{L}{2}$, considering the intersection $A\cap S_i$ for a given $i$, and applying a shift: When $\dim_{\rm H} A>1$, there is some choice of $i$ so that $\dim_{\rm H} (A\cap S_i) >1$; 
When $\dim_{\rm H} A\le 1$, a sequence of squares can be chosen so that the dimension of $A\cap S_i$ approaches the dimension of $A$, and a limiting process is used to complete the proof.  

We refer to Figure \ref{a11} for the proceeding notation.  For a $\lambda \in \mathbb{R},$ denote the vertical line
 $$\ell _\lambda:=\left\{(x,y):x=\lambda\right\}.$$
\begin{figure}[hhh]
  \centering
  \includegraphics[width=7cm]{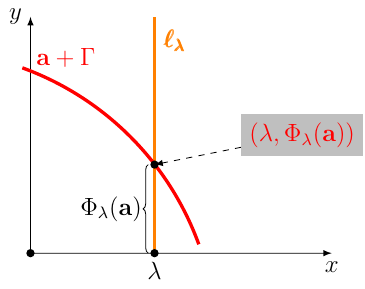}
\caption{$\Phi_\lambda(a)$}\label{a11}
\end{figure}
Define the one-parameter family of mappings $\{\Phi_\lambda(\mathbf{ \cdot })\}_{\lambda\in I}$,
$\Phi_\lambda:\Omega \to \ell _\lambda$ as the $y$-coordinate of the intersection 
$$
  \Phi_\lambda:a\mapsto \left\{\ell _\lambda\cap(a+\Gamma)\right\}.
$$
When the intersection is non-empty,  $\ell _\lambda\cap(a+\Gamma)$ is a singleton and its $y$-coordinate can written explicitly as
\begin{equation}\label{Phi}
  \Phi_\lambda(a)= a_2 + \ga(\la-a_1),
\end{equation}
where $a = (a_1, a_2)$. 
Note $\Phi_\lambda(a)$ is defined for $\la\in I$ provided that $\la-a_1 \in I$.

\begin{definition}\label{good_interval}[Good intervals]
A compact subinterval $\widehat{I}\subset I$ is a \textit{good interval for the family}
$\{\Phi_\la\}_{\la\in \widehat{I}}$ if $\Phi_\la(x)$ is defined for each $\la\in \widehat{I}$ and $x\in \Om$. 
This happens identically when $\la-x_1 \in I$ for each $x=(x_1, x_2) \in \Om$ and $\la \in \widehat{I}$.  
Provided $\widehat{I}$ is good, we have
\begin{equation}\label{slice_eq} 
\Phi_\la(A) = \left(A+ \Ga\right)\cap \ell_\la
\end{equation}
whenever $A\subset \Om$ and $\la \in \widehat{I}$. 
\end{definition}

Note that 
$\widehat{I}:= \left[ \frac{L}{2}, L\right]$
 is a good interval; indeed, 
 if $a=(a_1, a_2) \in \Om=\left[0, \frac{L}{2}\right]^2$ and we restrict $\la\in \left[ \frac{L}{2}, L\right]$, then $\la-a_1 \in I= [0,L]$ so that $\Phi_\la(a)$ is defined. 
Since $\Phi_\la(a)$ is defined for $a \in A$ and $\la \in \widehat{I} =  [\frac{L}{2}, L]$, 
 equation \eqref{slice_eq} combined with Fubini's theorem implies that
 
\begin{equation}\label{Fubini_eq}
 \int_{L/2}^{L}   \mathcal{H}^1\left( \Phi_\la(A) \right) d\la
=
 \int_{L/2}^{L}   \mathcal{H}^1\left( \left( A+\Ga\right) \cap \ell_\la \right) d\la
\le 
\pazocal{L}_2\left( A+ \Ga \right).
\end{equation}
This observation will be used to prove part (a) of Theorem \ref{main_measdim} and a similar Fubini-like equation introduced below will be used to prove part (b). 

Although it will not be necessary for the proof of Theorem \ref{main_measdim}, we note that by extending the curve $\Ga$ and defining a corresponding extension maps $\tilde{\Phi}_\la$, we can arrange matters so that the original parameter interval $I$ is good for $\{\tilde{\Phi}_\la\}_{\la \in I}$.  This extension will play a role in the proof of Theorem \ref{main_critical}, where we will require a reverse inequality to that in \eqref{Fubini_eq}.  
As this extension will rely on the notation above, we introduce it here in the remainder of the section.  

Observe that for $A\subset \Om$, we have $A+\Ga\subset [0,2L]^2$.  
We extend the curve $\Ga$ and introduce extension operators so that $\Phi_\la(a)$ is defined for each $\la \in [0,2L]$ and each $a\in A$.  
To this end, 
extend the function $\ga$ on $I$ to a function $\widetilde{\ga}$ on $\widetilde{I}=\left[ -\frac{L}{2},2L \right]$ so that 
$$\widetilde{\Ga}= \{(t, \widetilde{\ga}(t)) : t\in \widetilde{I} \}$$
is a good curve in the sense of Definition \ref{good_curve}; 
The details of this extension are given in \cite[Section 2.3]{DT22}.
For $x = (x_1, x_2) \in \Om=\left[0, \frac{L}{2}\right]^2$ and $\la \in [0,2L]$,  define 
\begin{equation}\label{Phi_extend}
\widetilde{\Phi}_\la(x) = x_2 + \widetilde{\ga}(\la-x_1).
\end{equation}
Now
$\widetilde{\Phi}_\la(x)$ agrees with $\Phi_\la(x)$ whenever $\la-x_1 \in I$.  
Note that for $x\in \Om$ and $\la\in [0,2L]$, it holds that $\la-x_1 \in \widetilde{I}$ so that 
$\widetilde{\Phi}_\la(x)$ is defined.  It follows that 
\begin{equation}\label{slice_extend_eq} 
\widetilde{\Phi}_\la(A) = \left(A+ \widetilde{\Ga}\right)\cap \ell_\la
\end{equation}
whenever $A\subset \Om$ and $\la \in [0,2L]$.  
Now, 
\begin{equation}\label{Fubini_extend_eq}
\pazocal{L}_2\left( A+ \Ga \right) 
\le \int_0^{2L}   \mathcal{H}^1\left( \left( A+\widetilde{\Ga}\right) \cap \ell_\la \right) d\la
= \int_0^{2L}   \mathcal{H}^1\left( \widetilde{\Phi}_\la(A) \right) d\la,
\end{equation}
where we used that $A+\Ga\subset \left( A+\widetilde{\Ga}  \right) \cap [0,2L]^2$ and Fubini's theorem in the first inequality,  and \eqref{slice_extend_eq} in the second.  
This inequality will be used in the proof of Theorem \ref{main_critical}

\subsection{The transversality method}\label{KT50}
Before stating the main tool, Corollary \ref{a23}, which will be used to prove Theorem \ref{main_measdim}, we give a brief introduction to the notion of transversal maps.  
The transversality method for Hausdorff dimension of the attractors of a one-parameter family of IFS (iterated function systems)  first appeared in the paper of the first author and M. Pollicott \cite{PoSi}. Then B. Solomyak
developed the transversality condition for the absolute continuity of the invariant measures for a one parameter family of IFS in \cite{Sol95}.
Moreover, B. Solomyak combined the methods from \cite{PoSi} and \cite{Sol95} in \cite{Sol98}
to establish a much more general transversality method for generalized projections (like $\Phi_\lambda( \cdot )$  above). The next step was made by Peres and Schlag \cite{PeSc00} to work out an even more general transversality method.  A slight modification of Peres and Schlag transversality appeared in R. Hovila,E.  J{\"a}rvenp{\"a}{\"a}, M. J{\"a}rvenp{\"a}{\"a}, F. Ledrappier \cite[Definition 2.4]{HovJ2Led}.  This result will be stated and used in Section \ref{g88}.

A general version of the transversality method was proved by B. Solomyak in \cite[Theorem 5.1]{Sol98}. An immediate corollary of this in our special case is as follows.  Here, we use the notation introduced above in Subsection \ref{r88}. 

\begin{corollary}[Solomyak-the transversality condition]\label{a23}
Let $\Gamma$ be a good curve, let $\widehat{I}$ be a good interval 
and assume that
$A \subset \Om$.
If hypotheses  \textbf{(H1)} and \textbf{(H2)} below hold:
  \begin{description}
  \item[(H1)] $  \|\Phi_\lambda(a)-\Phi_\lambda(b)\|<c \cdot \|a-b\|$, $\forall a,b\in \Om$ and
  \item[(H2)]  $
  \pazocal{L}_1\left\{
  \lambda\in \widehat{I}:
  \frac{\|\Phi_\lambda(a)-\Phi_\lambda(b)\|}
  {\|a-b\|}
  <r
  \right\}
   \leq c \cdot r
  $ for all $r>0$, $\forall a,b\in \Om$ with $a\neq b$,
\end{description}then
\begin{description}
  \item[(i)] if $\dim_{\rm H} (A)>1$, then $\pazocal{L}_1(\Phi_\lambda(A))>0$
  for $\pazocal{L}_1$ almost all $\lambda\in \widehat{I}$;
  \item[(ii)]
  if $\dim_{\rm H} (A) \leq 1$, then $\dim_{\rm H} (\Phi_\lambda(A))=\dim_{\rm H} (A)$ for $\pazocal{L}_1$ almost all $\lambda\in \widehat{I}$.
\end{description}
Here, the constant $c>0$ denotes a positive number which is  independent of $a,b,r$ and $\la$.
\end{corollary}

Hypothesis \textbf{(H2)} is called the "\emph{transversality condition}".
Hypothesis  \textbf{(H1)} means that $\Phi_\lambda$ is a Lipschitz mapping, and a simple computation shows that  \textbf{(H1)} holds for each $\lambda \in \widehat{I}:=[\frac{L}{2},L]$.

We now turn to completing the proof of Theorem \ref{main_measdim} using Corollary \ref{a23}.  
We momentarily assume that the maps $\Phi_\al$ defined in \eqref{Phi} satisfy the transversality condition of \textbf{(H2)}, and delay verifying transversality until Subsection \ref{checking_section} below.

\subsection{Proof of both parts (a) and (b)  of Theorem \ref{main_measdim}
with the transversality method }\label{main_measdim_sec}
 We verify that \textbf{(H2)} holds in Section \ref{checking_section}. To complete the proof of Theorem \ref{main_measdim}, we verify the following Lemma.
\begin{lemma}\label{a24}
The conclusions \textbf{(i)} and \textbf{(ii)} of Corollary \ref{a23}
imply that the conclusions (a) and (b)  of Theorem \ref{main_measdim}.
\end{lemma}
\begin{proof}
Part \textbf{(a)} of of Theorem \ref{main_measdim} is now an immediate consequence of  \textbf{(i)} in light of \eqref{Fubini_eq}.

We now prove that \textbf{(b)} follows from  \textbf{(ii)}.  Assume that
  $\dim_{\rm H} (A) \leq 1$. 
  First, we give the upper bound:
  \begin{equation}\label{a22}
    \dim_{\rm H} (A+\Gamma) \leq \dim_{\rm H} A +1.
  \end{equation}
Using \cite[Corollary 7.4]{Falc90} and observing that the box dimension of $\Gamma$ equals $1$, denoted $\dim_{\rm B}\Gamma=1 $, it follows that
  \begin{equation}\label{a21}
    \dim_{\rm H} (A\times \Gamma)=\dim_{\rm H} A+1.
  \end{equation}
Define $\Psi:\mathbb{R}^2\times \mathbb{R}^2\to \mathbb{R}^2$ by
  $\Psi(x,y):=x+y$. The fact that $\Psi$ is a Lipschitz mapping and \eqref{a21} yields that
  \begin{equation}\label{Lip}
  \dim_{\rm H} (A+\Gamma)=\dim_{\rm H} (\Psi(A,\Gamma))
   \leq  \dim_{\rm H} (A\times \Gamma) \leq \dim_{\rm H} A+1.
  \end{equation}
 This completes the proof of the upper bound \eqref{a22}.

Now we demonstrate that \textbf{(ii)} implies
  \begin{equation}\label{a20}
    \dim_{\rm H} (A+\Gamma) \geq 1+\dim_{\rm H} A.
  \end{equation}
  This immediately follows from \cite[Theorem 5.8]{Falc86}.
Namely,
\begin{theorem}[Fubini-like theorem]\label{092}
Let $A\subset \mathbb{R}^2$ and $U$ be a subset of the $y$-axis. We write
$A^y:=\left\{a: (a,y)\in A\right\}$. Suppose that for any $y\in U$ we have
$\mathcal{H}^t(A^y)>c$, for some constant $c$. Then
\begin{equation}\label{091}
  \mathcal{H}^{t+s}(A) \geq bc\mathcal{H}^s(U),
\end{equation}
where $b=b(s,t)$.
\end{theorem}
We use this theorem  with the following substitution: Set $\widehat{I} = \left[ \frac{L}{2}, L\right]$, the good interval of the $y$-axis introduced immediately following Definition \ref{good_interval}. 
Then by Corollary \ref{a23} part \textbf{(ii)}, for almost every $\la \in \widehat{I}$, 
$$\dim_{\rm H} \Phi_\lambda(A)=\dim_{\rm H} A,$$
Letting $t<\dim_{\rm H} A$ arbitrary, it follows that $\mathcal{H}^t(\Phi_\lambda(A))>1$ for almost every $t \in \widehat{I}$. 
For $\lambda\in \widehat{I}$, set
$$E:=\bigcup_{\lambda\in \widehat{I}}\Phi_\lambda(A) \, \, \text{ and } \,\, E^\lambda=\Phi_\lambda(A).$$ 
By Theorem \ref{092}, it follows that $\mathcal{H}^{1+t}(E) \geq b \mathcal{H}^1(\widehat{I} \,\, )>0$, which implies that
$\dim_{\rm H} E \geq 1+t$. 
As $t<\dim_{\rm H} A$ was arbitrary, it follows that 
$$
\dim_{\rm H} (A+\Gamma) \geq \dim_{\rm H} (E) \geq 1+\dim_{\rm H} A.
$$
\end{proof}

\subsection{Checking the transversality condition}\label{checking_section}
In this section, we verify that hypothesis \textbf{(H2)} of Corollary \ref{a23} is satisfied for the family $\{\Phi_\la\}_{\la\in \widehat{I}}$ whenever $\Ga$ is a good curve and $\widehat{I}$ is a good interval. 

\begin{proof}
Fix a choice of $\mathbf{a}=(a_1,a_2),\mathbf{b}=(b_1,b_2)\in \Om$ with $\mathbf{a} \neq \mathbf{b}$. The proof comes in two parts: the translated graphs $(\mathbf{a} + \Gamma)$ and $(\mathbf{b} + \Gamma)$ will either intersect at a point, or they will be disjoint. We first handle the intersecting case when
\begin{equation}\label{007}
  (\mathbf{a}+\Gamma)\cap(\mathbf{b}+\Gamma)\ne \emptyset .
\end{equation}
That is, suppose there exist $s_0,t_0\in I$ and $\mathbf{a}=(a_1,a_2)\in \R^2$ such that
$$
  \mathbf{x}:=(a_1,a_2)+(s_0, \gamma(s_0)) = (b_1,b_2)+ (t_0, \gamma(t_0)).
$$

\begin{figure}[h]
  \centering
  \includegraphics[width=12.5cm]{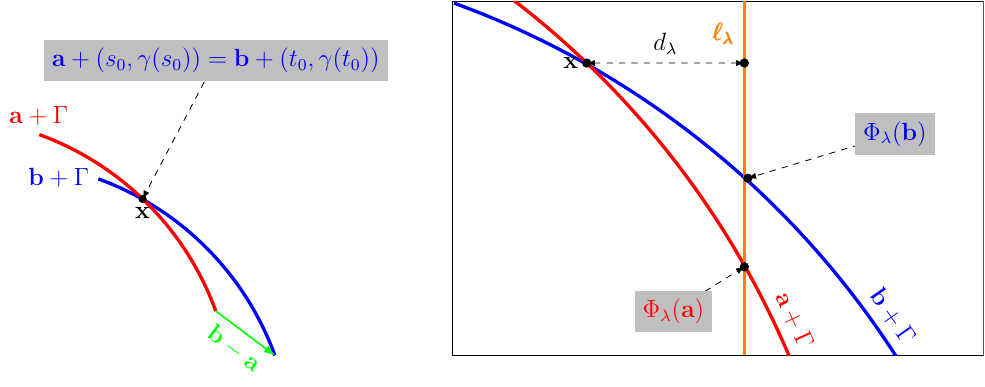}
  \caption{}\label{figure_2_c_a-vertical}
\end{figure}
Comparing coordinates, we have
 \begin{equation}\label{eq1}
 x_1= a_1  + s_0 = b_1 + t_0
 \end{equation}
 and 
 \begin{equation}\label{eq2}
 x_2 = a_2 + \gamma(s_0) = b_2 + \gamma(t_0).
  \end{equation}

For $\la \in \widehat{I} $, set
\begin{equation}\label{distance_la}
d_\la : = \mathrm{dist}(\mathbf{x},\ell _\lambda)=|\lambda - x_1|,\end{equation}
as depicted in Figure \ref{figure_2_c_a-vertical}. 
We verify that
\begin{equation}\label{transv_goal}
|\Phi_\lambda(\mathbf{a})-\Phi_\lambda(\mathbf{b})|
\sim
d_\lambda \cdot |\mathbf{a}-\mathbf{b}|
\end{equation}
where the implied constant is independent of $\lambda$, $\mathbf{a}$, and $\mathbf{b}$. Strictly speaking, we only need that the left hand side dominates the right hand side. Upon establishing equation \eqref{transv_goal}, it will follows that if $\de>0$ and $\lambda\in A$ satisfy
$|\Phi_\lambda(\mathbf{a})-\Phi_\lambda(\mathbf{b})|\le \de$, then
$$d_\lambda \cdot |\mathbf{a}-\mathbf{b}|\lesssim \de,$$
and so
\begin{equation}\label{1-trans}
\pazocal{L}_1 \{\la\in \widehat{I}: |\Phi_\lambda(\mathbf{a})-\Phi_\lambda(\mathbf{b})|\le \de\}
\lesssim \frac{\de}{|\mathbf{a}-\mathbf{b}|}
\end{equation}
which is the desired transversality condition.


We have two further reductions. First, as depicted in Figure \ref{figure_2_c_a-vertical}, we consider the case when $\la \geq x_1$ so that
\begin{equation}\label{distance_la_updated}
d_\la =\la - x_1 \geq 0.
\end{equation}
Note that the case when when $\la - x_1 < 0$ can be handled by reflecting $A$ and $\Ga$ about the $y$-axis. Second, by relabeling $\mathbf{a}$ and $\mathbf{b}$ if necessary, we may assume that when $\la> x_1$, it holds that
\begin{equation}\label{positivity_cond_phi} 
\Phi_{\lambda}(\mathbf{b}) - \Phi_{\lambda}(\mathbf{a})>0
 \end{equation}
as in Figure \ref{figure_2_c_a-vertical}. Finally, we will also have
\begin{equation}\label{positivity_cond_space}
(b_1-a_1)>0.
\end{equation}
This follows from the geometry of the curves: in order for \eqref{positivity_cond_phi} to hold in the intersecting case, the concavity of $\Gamma$ shows that $\mathbf{b}$ must lie below and to the right of $\mathbf{a}$. 

Using the concavity condition \eqref{concave}, we show that 
\begin{equation}\label{transv_goal_mini}
 \Phi_{\lambda}(\mathbf{b}) - \Phi_{\lambda}(\mathbf{a}) \sim (b_1-a_1)\cdot d_\la.
 \end{equation} Observe that by the bound on $\gamma'$ and the relationships established in \eqref{eq1} and \eqref{eq2},
\begin{equation}
|b_2-a_2| = |\ga(s_0) - \ga(t_0)| \le |s_0-t_0| = |b_1-a_1|.
\end{equation}
As such, proving \eqref{transv_goal_mini} will be sufficient to establish \eqref{transv_goal}. We now carry out the verification of \eqref{transv_goal_mini} in three cases based on the relative sizes of $d_{\lambda}$ and $|b_1 - a_1|$. We will handle the non-intersecting case (where \eqref{007} does not hold) separately.

\noindent
\textbf{Case 1:} $(b_1 - a_1) < \frac{d_{\lambda}}{2}$. 
We begin by examining the simplest case, which motivates the finer analysis to come. This is depicted in the following figure:
\begin{figure}[ht]
  \centering
  \includegraphics[width=12cm]{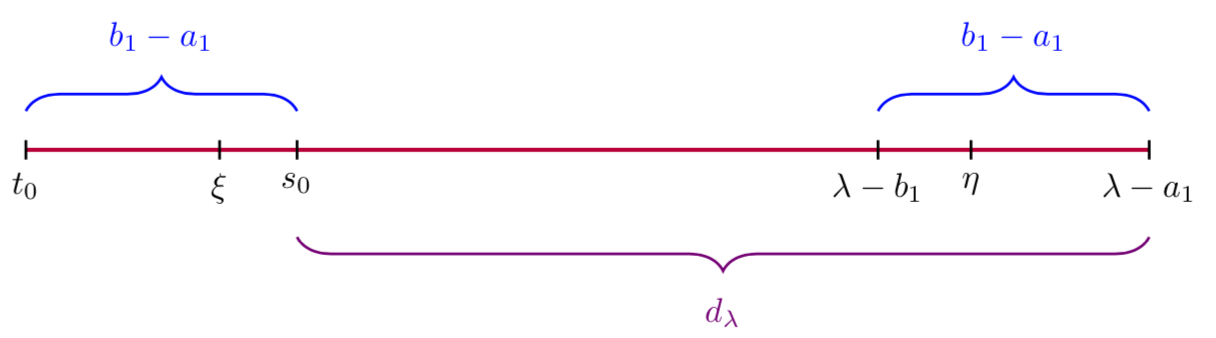}
  \caption{Case 1}
  \label{figure_case1}
\end{figure}

Using the definition of $\Phi_\la$ in \eqref{Phi}, the relationships established in \eqref{eq1} -- \eqref{eq2}, and the mean value theorem, we have
\begin{align*}
\Phi_{\lambda}(\mathbf{b}) - \Phi_{\lambda}(\mathbf{a}) 
&= \left( b_2 +  \gamma(\lambda-b_1)\right)  -    \left( a_2 +  \gamma(\lambda-a_1) \right)  \\
&=     \left(  b_2 -a_2\right)     + \left( \gamma(\lambda-b_1) - \gamma(\lambda-a_1)\right)   \\
&=      \left(  \gamma(s_0) - \gamma(t_0)\right)  + \left(  \gamma(\lambda-b_1) - \gamma(\lambda-a_1)\right)  \\
&=   \gamma'(\xi)(b_1 - a_1) -   \gamma'(\eta)(b_1 - a_1) \\
&= [\gamma'(\xi) - \gamma'(\eta)](b_1 - a_1),
\end{align*}
for some $\eta \in (\lambda- b_1,\lambda - a_1)$ and $\xi\in (t_0, s_0)$.

It follows by \eqref{concave} that
$$\Phi_{\lambda}(\mathbf{b}) - \Phi_{\lambda}(\mathbf{a})  \sim 
 (\eta - \xi) \cdot (b_1- a_1).$$
Since $(b_1-a_1)< d_\lambda/2$, we see that  \eqref{transv_goal_mini} is verified following the observation that
$$(\eta - \xi ) \sim d_\lambda.$$
To see this, recall from \eqref{distance_la_updated} and \eqref{eq1} that 
$d_\la = \la- a_1 - s_0 = \la - b_1 -t_0$. Following Figure \ref{figure_case1},
$$\eta - \xi > (\lambda - b_1) -s_0 = (\lambda- b_1 - t_0) - (s_0 - t_0) = d_\lambda - (b_1 - a_1),$$
and similarly  
$$\eta - \xi < (\lambda-a_1) -t_0 = (\lambda-a_1-s_0) + (s_0 -t_0) = d_\lambda + (b_1 - a_1).$$

Before moving to the general argument, we observe that the separation of $d_\la$ and $(b_1-a_1)$ was crucial in guaranteeing that the variables arising from the application of the mean value theorem, $\xi$ and $\eta$, were properly separated.  More generally, a finer analysis using telescoping sums is used to guarantee such separation.

\noindent
\textbf{Case 2:} $\frac{d_{\lambda}}{2} \le (b_1 - a_1 )< d_{\lambda}$. 
Set 
\begin{equation}\label{p_k}
p = \frac{b_1- a_1}{2} , \,\ \text{ and } \,\, 
q= s_0.
\end{equation}

First, 
we take a moment to compare the variables under examination. 
Note $p>0$ by \eqref{positivity_cond_space}. 
Using \eqref{distance_la_updated} and \eqref{eq1}, we can write $d_\la = \la - b_1 - t_0$ and $b_1- a_1 = s_0 - t_0$. Therefore, when $b_1 - a_1 < d_{\lambda}$, then $s_0-t_0 < \la -b_1 - t_0$ and so
$s_0<\la - b_1.$  
This implies that
$$
t_0 < s_0 < \lambda - b_1 < \lambda -a_1,
$$
and so for $p$ and $q$ as in \eqref{p_k}, 
$$
t_0 = q-2p < q-p< q= s_0 < \lambda - b_1 = \la - a_1-2p < \la - a_1 - p < \lambda -a_1. 
$$
Appealing to \eqref{eq2}, we can write 
\begin{align*}
\Phi_{\lambda}(\mathbf{a}) - \Phi_{\lambda}(\mathbf{b}) 
&=\gamma(\lambda-a_1) -\gamma(\lambda-b_1)  -    \left( b_2 -a_2  \right)  \\
&=   \gamma(\lambda-a_1) - \gamma(\lambda-b_1)  -  \left( \gamma(s_0) - \gamma(t_0) \right)  \\
&= \sum_{j=0}^1 \left( \ga(\la - a_1-jp) -\ga(\la - a_1 - (j+1)p) \right) \\
&-  \sum_{j=0}^1 \left( \ga(q- jp) -\ga(q-(j+1)p) \right).  
\end{align*}
Applying the mean value theorem, there exists $h_0, h_1, h_0', h_1' \in (0,1)$ so that 
\begin{align*}
\Phi_{\lambda}(\mathbf{a}) - \Phi_{\lambda}(\mathbf{b}) 
&= \sum_{j=0}^1 \left( \ga'(\la - a_1-jp-h_jp)\cdot p  \right) \\
&-  \sum_{j=0}^1 \left( \ga'(q- jp-h_j' p ) \cdot p  \right),
\end{align*}
and it follows that 
\begin{equation}\label{main_case2}
\Phi_{\lambda}(\mathbf{a}) - \Phi_{\lambda}(\mathbf{b}) 
\sim \left( \sum_{j=0}^1 \left( \ga'(\la - a_1-jp-h_jp)
-  \ga'(q- jp-h_j' p ) \right) \right)  \cdot p.
\end{equation}

The purpose for adding and subtracting terms, is that the terms 
$\la - a_1-jp-h_jp$ and $q- jp-h_j' p$ are now appropriately separated for $j=0,1$.
Indeed, when $d_\lambda  >(b_1 - a_1)$,  recalling that $q=s_0$, 
it holds that 
$$(\la - a_1-jp-h_jp)-(s_0- jp-h_j' p)
=d_\la -h_jp + h_j'p
\geq d_\la/2, $$
and 
$$(\la - a_1-jp-h_jp)-(s_0- jp-h_j' p)
=d_\la -h_jp + h_j'p
\le 3d_\la/2. $$

The key point is that in \eqref{main_case2}, the arguments of $\gamma'$ within each summand are separated by a positive quantity comparable to $d_{\lambda}$. Using the bi-Lipschitz condition on $\gamma'$ (in which case $\gamma'$ is strictly monotonic on $I$), we conclude that
$$\Phi_{\lambda}(\mathbf{b}) - \Phi_{\lambda}(\mathbf{a}) \sim  d_\lambda \cdot p.$$ 
Since $p \sim (b_1 - a_1)$, this case is completed.

\noindent
\textbf{Case 3:} $d_{\lambda} \le (b_1 - a_1)$. 
Set
\begin{equation}\label{pq}
p = \frac{ d_\lambda}{2} , \,\ \text{ and } \,\, 
q= (\la - b_1).
\end{equation}

With this choice of $p$ and $q$, the proof proceeds as in the  previous case. This situation is depicted below in Figure \ref{figure_case3}.

\begin{figure}[ht]
  \centering
  \includegraphics[width=12cm]{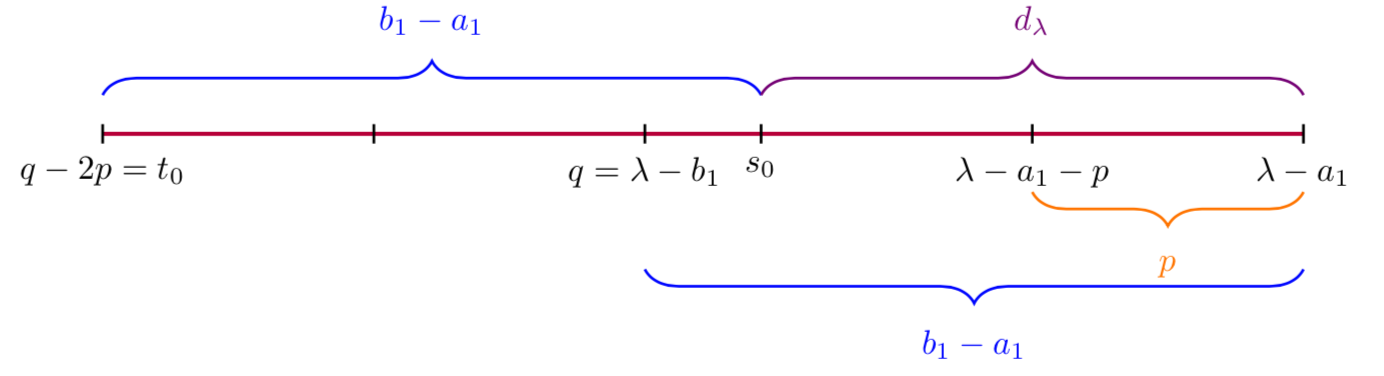}
  \caption{Case 3}
  \label{figure_case3}
\end{figure}

Using \eqref{eq1} and \eqref{distance_la_updated}, we can write  $d_\la = \la - b_1 - t_0 = \la - a_1 - s_0\geq 0$ and $b_1- a_1 = s_0 - t_0 > 0$. Therefore, when $d_\lambda \le b_1 - a_1$, then 
$ \la - b_1 - t_0 \le s_0-t_0$ and so
$\la - b_1 \le s_0$. Combining these observations, if $d_\la \le (b_1-a_1) $, then 
$$ t_0 \le  \lambda - b_1 \le s_0 \le   \lambda -a_1,
$$
and so, for $p$ and $q$ as in \eqref{pq}, 
$$t_0= q-2p \le  q-p \le q= \lambda - b_1 \le s_0 = \la - a_1 - 2p \le \la - a_1 - p \le  \lambda -a_1. 
$$

Using an identical telescoping argument as that used in the previous case to obtain \eqref{main_case2}, except now with $p$ and $q$ as in \eqref{pq}, we conclude that there exists $h_0, h_1, h_0', h_1' \in (0,1)$ so that
\begin{equation}\label{common2}
\Phi_{\lambda}(\mathbf{a}) - \Phi_{\lambda}(\mathbf{b}) 
\sim \left( \sum_{j=0}^1 \left( \ga'(\la - a_1-jp-h_jp)
-  \ga'(q- jp-h_j' p ) \right) \right)\cdot p.
\end{equation}

We now observe that 
$\la - a_1-jp-h_jp$ and $q- jp-h_j' p$ are sufficiently separated for $j=0,1$
when $d_\lambda \le b_1 - a_1$:
$$(\la - a_1-jp-h_jp)-(q- jp-h_j' p)
= b_1-a_1 -h_jp + h_j'p
\geq  (b_1-a_1)/2, $$
and 
$$(\la - a_1-jp-h_jp)-(q- jp-h_j' p)
= (b_1-a_1) -h_jp + h_j'p
\le 3(b_1-a_1)/2. $$

As in Case 2 above, we have now established the necessary separation between the arguments of $\gamma'$ in each summand; it follows that
$$\Phi_{\lambda}(\mathbf{b}) - \Phi_{\lambda}(\mathbf{b})\sim (b_1 -a_1)\cdot p.$$
Since $p \sim d_{\lambda}$, this case is finished.

%
%

\noindent \textbf{Non-intersection case:} It remains to verify  \textbf{(H2)} of Corollary \ref{a23} when \eqref{007} does not hold. 
We verify the equivalent statement \eqref{1-trans}. 
Assume that $\mathbf{a}$ and $\mathbf{b}$ are such that 
\begin{equation}\label{emptycase}
  (\mathbf{a}+\Gamma)\cap(\mathbf{b}+\Gamma)= \emptyset.
\end{equation}
Let $\de>0$.  
For each $\la\in \widehat{I}$, set
\begin{align*}
h(\la) 
&:= \Phi_{\lambda}(\mathbf{b}) - \Phi_{\lambda}(\mathbf{a}) 
=\ga(\la-b_1) - \ga(\la-a_1) + (b_2-a_2). 
\end{align*}

Relabeling if necessary, we may assume that the graph $(\mathbf{b}+\Gamma)$ is above $(\mathbf{a}+\Gamma)$ in the sense that 
for each $\la\in \widehat{I}$, it holds that
$$
h(\la) >0.
$$

Observe that in the case that $a_1= b_1$, then 
$h(\la) = b_2-a_2$ is constant, and so the left-hand-side of \eqref{1-trans} is non-zero identically when 
$h(\la)= b_2-a_2 = |a-b| \le \de$, in which case the right-hand-side of \eqref{1-trans} is bounded below by the constant $c$, and  the inequality is satisfied provided that $c$ is chosen so that $c\geq |A|$.  

Assume then that $a_1 \neq b_1$.  
We will apply a vertical shift to the curve $(\mathbf{b}+ \Gamma)$ to reduce to the intersection case considered in \eqref{007} and handled above. 
It is a consequence of the assumption in Definition \ref{good_curve} that $\ga''\neq 0$ that there exists a unique $\widehat{\la}\in A$ where $h(\la)$ is minimized.  Set 
$$d:= h(\widehat{\la}).$$ (Indeed, when $a_1\neq b_1$, note that $h$ is strictly monotonic as $h'\neq 0$ since $\ga''\neq 0$.) 
Now
$$\left(  \Ga + (b_1, b_2-d) \right) \cap \left( \Ga + \mathbf{a}\right) \neq \emptyset,$$
and we see that 
\begin{equation}\label{shift}
\Phi_{\la}(\mathbf{b}) =b_2    + \gamma(\lambda -b_1)    = b_2 - d   + \gamma(\lambda -b_1)  + d 
=  \Phi_{\la}((b_1, b_2 - d )) + d.
\end{equation}

Set $ \mathbf{b}(d)=    (b_1, b_2 - d )$. 
Now, if $\la$ is such that $h(\la) =  \Phi_{\lambda}(\mathbf{b}) - \Phi_{\lambda}(\mathbf{a}) \le \de$, then 
$ \Phi_{\la}(   \mathbf{b}(d)  ) - \Phi_{\lambda}(\mathbf{a}) \le \de-d \le \de$.
Note we may assume that $\de\geq d$ since $h(\la) \geq d$ for each $\la\in \widehat{I}$.  
Therefore
\begin{equation}\label{containment}
\{ \la\in \widehat{I}:  \Phi_{\lambda}(\mathbf{b}) - \Phi_{\lambda}(\mathbf{a}) \le \de\}
\subset
\{ \la\in \widehat{I}:  \Phi_{\lambda}((b_1, b_2 - d )) - \Phi_{\lambda}(\mathbf{a}) \le \de\},
\end{equation}
and it follows from the the previous Cases 1-3 that there exists a constant $c>0$ that depends only on the constant $\La$ in \eqref{concave} so that 
\begin{equation}\label{bd}
\pazocal{L}_1\{ \la\in \widehat{I}:  \Phi_{\lambda}((b_1, b_2 - d )) - \Phi_{\lambda}(\mathbf{a}) \le \de\} \le \frac{c\,\de}{|\mathbf{b}(d)-\mathbf{a}|}.
\end{equation}

Combining \eqref{containment} and \eqref{bd}, we see that if $|\mathbf{b}(d)-\mathbf{a}|$ were bounded below by $|\mathbf{b}-\mathbf{a}|$, then the argument would be complete. Since this may not always be the case, we need a slightly more delicate analysis.  

We will now proceed in two cases, based on the relative sizes of $|b_1 - a_1|$ and $|b_2 - a_2|$. When the first difference is dominant, the shift between $b$ and $a$ is mostly horizontal and this horizontal translation is detected by the first coordinate of $\mathbf{b}(d)$. The more challenging case is when the translation is nearly vertical; this will follow the same lines as when $b_1 = a_1$. To be precise, we now consider the cases when 
$ |b_1-a_1| \geq \frac{1}{2} |b_2-a_2|$ and 
$ |b_1-a_1| < \frac{1}{2} |b_2-a_2|$ separately. 

In the former case, 
$$|b_1- a_1| \gtrsim |b-a|$$
and so 
\begin{align*}
|\mathbf{b}(d) - \mathbf{a}|^2
&= |b_1-a_1|^2  + |b_2-d-a_2|^2\\
&\geq  |b_1-a_1|^2\\
&\gtrsim |b-a|^2.
\end{align*}
In this case, we see that if $|\mathbf{b}(d)-\mathbf{a}|$ is bounded below by a constant multiple of $|\mathbf{b}-\mathbf{a}|$, and the argument is complete upon combining \eqref{containment} and \eqref{bd}.

Now consider the latter case that $ |b_1-a_1| < \frac{1}{2} |b_2-a_2|.$ Suppose that $\la$ is such that $h(\la) \le \de$.  
By the mean value theorem, there exists an $\eta$ so that 
$$ \ga(\la-b_1) - \ga(\la-a_1)=  -\ga'(\eta) (b_1-a_1).$$
Since $\Ga$ was assumed to be a good curve in the sense of Definition \ref{good_curve} and $\widehat{I}$ is a good interval (see the assumption at the top of Section \ref{checking_section}), 
we have
$|\ga'(t)| \le 1$. 
It follows from the reverse triangle inequality that
\begin{align*}
h(\la) &\geq |b_2-a_2| - |\ga'(\eta) (b_1-a_1)|\\
&\geq |b_2-a_2| - | b_1-a_1|\\
&\geq |b_2-a_2| -\frac{1}{2} |b_2-a_2|\\
&= \frac{1}{2} |b_2-a_2|\\
&\sim  |\mathbf{b}  -\mathbf{a} |,
\end{align*}
where the implicit constants are independent of $\mathbf{b}$, $\mathbf{a}$ and $\la$.  
It follows that there exists a $c'>0$ so that if  $\la$ is such that $h(\la) \le \de$, then 
$|\mathbf{b}  -\mathbf{a} | \le c'\de$ or $1 \le \frac{ c'\de  }{   |\mathbf{b}  -\mathbf{a} |}.$ Now, 
$$|\{\la\in \widehat{I}: h(\la) \le \de\} | \le |A| \le c \le c \frac{c' \, \de}{ |\mathbf{b}  -\mathbf{a} |},$$ provided $c$ is chosen so that $c\geq |A|$.   
\end{proof}


\section{Proof of Theorem \ref{main_critical} }\label{g88}
We first prove Theorem \ref{main_critical} in the special case when the curvature does not vanish at any point. 
\subsection{The case when the curvature does not vanish}
The following is a restatement of Theorem \ref{main_critical} in the special case that the curvature of $\Ga$ does not vanish.  
\begin{theorem}[Simplified variant of Theorem \ref{main_critical}]\label{g98}
Let $\Gamma \subset \mathbb{R}^2$ be a simple $\ \mathcal{C}^2$ curve of positive and finite length such that the curvature of \  $\Gamma$ does not vanish at any point.
Let $A\subset \mathbb{R}^2$ be Borel measurable with $\dim_{\rm H} A=1$. Further we assume that $\mathcal{H}^1|_A$ is $\sigma$-finite.
Then,
$\pazocal{L}_2(A+\Gamma)=0$ if and only if for every
rectifiable curve $\gamma$, we have $\mathcal{H}^1(\gamma\cap A)=0$.
\end{theorem}

The proof of Theorem \ref{g98} follows from the transversality of the family $\left\{\Phi_\lambda\right\}$, which we already verified in Section \ref{checking_section}, coupled with the following theorem due to
R. Hovila, E. J{\"a}rvenp{\"a}{\"a}, M. J{\"a}rvenp{\"a}{\"a}, F. Ledrappier
 \cite[Theorem 1.2]{HovJ2Led}. 
\begin{theorem}[Hovila, J{\"a}rvenp{\"a}{\"a},  J{\"a}rvenp{\"a}{\"a}, Ledrappier]\label{006}
  Let $A \subset \mathbb{R}^n$ be $\mathcal{H}^m$-measurable with $\mathcal{H}^m(A)<\infty $. Assume that $\Lambda \subset \mathbb{R}^\ell $
  is open and $\left\{P_\lambda:\mathbb{R}^n\to\mathbb{R}^m\right\}_{\lambda\in\Lambda}$ is a transversal family of maps. Then $A$ is purely unrectifiable, if and only if $\mathcal{H}^m(P_\lambda(A))=0$
  for $\pazocal{L}_\ell $-almost all $\lambda\in\Lambda$.
\end{theorem}

We apply Theorem \ref{006} with $n=2$, $m=1$, $\ell =1$ and $\left\{P_\alpha\right\}=\left\{\Phi_\alpha\right\}$.
The proof of Theorem \ref{g98} is really quite simple from here; the basic idea is to use Theorem \ref{006} coupled with either and \eqref{Fubini_eq}  or \eqref{Fubini_extend_eq} to show $\pazocal{L}_2\left( A+ \Ga \right)>0$ when $A$ is regular 
and $\pazocal{L}_2\left( A+ \Ga \right)=0$ when $A$ is irregular.  We now give the details.  

\begin{proof}[Proof of Theorem \ref{g98} ]
It follows from the comments in Section \ref{r88} that, without loss of generality, we may assume (using the notation in that section) that
\begin{equation}\label{r84}
  \Gamma \mbox{ is a good curve, }\, \,  I= [0,L], \, \, \Ga \subset I^2 = [0,L]^2,  \, \, A \subset \Om:= \left[0, \frac{L}{2}\right]^2. 
\end{equation}
We first handle the case when $0<\mathcal{H}^1(A)<\infty$.
Then to verify Theorem \ref{g98}, it is enough to prove both of the following

\begin{equation}
   A \mbox{ is irregular} \Longrightarrow
 \pazocal{L}_2(A+\Gamma)=0,
\end{equation}
and 
\begin{equation}
A \mbox{ is regular } \Longrightarrow
 \pazocal{L}_2(A+\Gamma)>0.
\end{equation}

We appeal to Theorem \ref{006}, \eqref{Fubini_eq} and \eqref{Fubini_extend_eq} to complete the proof.  
Given $x\in \Om$, recall
that $\Phi_\la(x)$ is defined for each $\la\in [0,\frac{L}{2}]$ and that $\widetilde{\Phi}_\al$, the extension operator discussed in Section \ref{r88},  
is defined for each $\la\in [0,2L]$.
In the case that $A$ is irregular, then $\pazocal{L}_1(\widetilde{\Phi}_\lambda(\mathbf{A}))=0$ by Theorem \ref{006} for almost all $\lambda\in [0, 2L]$.  It follows by Fubini then that $\pazocal{L}_2\left(\bigcup_{\lambda} \widetilde{\Phi}_\lambda(A)\right)=0$, where the union is taken over $\lambda\in [0,2L]$.  We observe that $$A + \Gamma \subset \bigcup_{\lambda \in [0, 2L]} \widetilde{\Phi}_\lambda(A),$$ and so $\pazocal{L}_2(A+\Gamma)=0$, as in \eqref{Fubini_extend_eq}. 

Assume that $A$ is regular, then $\pazocal{L}_1( \Phi_\lambda(\mathbf{A}))>0$ by Theorem \ref{006} for almost all $\lambda\in \left[0, \frac{L}{2} \right]$, and by Fubini $\pazocal{L}_2(\bigcup_{\lambda} \Phi_\lambda(A) ) >0$.    We observe that $$\bigcup_{\lambda   \in \left[0, \frac{L}{2} \right]} \Phi_\lambda(A) \subset A +\Gamma,$$ and so $\pazocal{L}_2(A+\Gamma)>0$, as in \eqref{Fubini_eq}. 

This establishes Theorem \ref{g98} when $0<\mathcal{H}^1(A)<\infty$.  
When $\mathcal{H}^1(A)=0$, then $\mathcal{H}^1(A\cap \ga)=0$ for each rectifiable curve $\ga$, and the proof above applies to show that $\pazocal{L}_2(A+\Ga) = 0$.
The case when $\mathcal{H}^1(A)=\infty$ follows from the 
assumption that $\mathcal{H}^1|_A$ is $\sigma$-finite combined with the dominated convergence theorem. 
\end{proof}

The way we used that the curvature does not vanish at any point is as follows: To establish transversality, we needed to assume this condition in Corollary \ref{a23}. 
We remark that the definition of a transversal family in \cite{HovJ2Led} is slightly different from the one we gave in Corollary \ref{a23}, however it is easy to see that hypotheses \textbf{(H1)} and \textbf{(H2)} in Corollary \ref{a23}
imply that the conditions of the transversal family \cite[Definition 2.4]{HovJ2Led} hold.

\subsection{The general case}\label{general_case}
Now we prove Theorem \ref{main_critical} in the general case using Theorem \ref{g98} and the following Lemma.
\begin{lemma}\label{r99}
Let $\Ga= \{(t, \ga(t)) : t\in I\}$ be a good curve in the sense of Definition \ref{good_curve}. 
Let $Z \subset I$ be such that $\pazocal{L}_1(Z)=0$. We define
$$\Gamma_Z:=\ga(Z)=\left\{x\in\Gamma:\exists t\in Z, x=\ga(t)\right\}.$$
Assume that  $A \subset \mathbb{R}^2$ satisfying $\mathcal{H}^1(A)<\infty $. Then we have
\begin{equation}\label{r98}
  \pazocal{L}_2\left(A+\Gamma_Z\right)=0.
\end{equation}
\end{lemma}

With Lemma \ref{r99} in tow, we are ready to prove the general theorem.  The proof of Lemma \ref{r99} is delayed to the end of this section.

\begin{proof}[Completion of the Proof of Theorem \ref{main_critical}]
Assume that $A$ is an irregular $1$-set.
Let  $\Gamma \subset \mathbb{R}^2$ be a $\mathcal{C}^2$ curve such that the curvature of $\Gamma$ does not vanish except at a set of points having zero $\mathcal{H}^1$-measure. 
 Assume that $\mathbf{r}(t)$, $t\in I$ is the arc-length parametrization of $\Gamma$. Let $\chi(t)$
be the curvature of $\Gamma$ at the point  $\mathbf{r}(t)$. By assumption,  the set
$
Z:=\left\{t\in I:
\chi(t)=0\right\}
$
is such that
\begin{equation}\label{r82}
\pazocal{L}_1(Z)=0.
\end{equation}

 First assume that $Z$ is compact.
Then the complement of $Z$,  $G:=Z^c$ can be written as a countable disjoint union $G=\bigcup\limits_{n=1}^{\infty }G_n$ of open subintervals $G_n \subset \mathbb{R}$.  Then we represent every $G_n=\bigcup\limits_{m=1}^{\infty }I_{m}^{(n)}$ for closed intervals
$I_{m}^{(n)}$.  That is
\begin{equation}\label{r81}
  I=Z \, \bigcup \, \left( \bigcup\limits_{n,m=1}^{\infty }  \left(I \cap {I}_{m}^{(n)} \right) \right).
\end{equation}
Then for
$$
\Gamma_Z:=\mathbf{r}(Z) \mbox{ and }
\Gamma_{m}^{(n)}:=\mathbf{r}\left(I \cap I_{m}^{(n)}\right),
$$
 we have
 \begin{equation}\label{r80}
   \Gamma=\Gamma_Z \, \bigcup \, \left( \bigcup\limits_{n,m=1}^{\infty }\Gamma_{m}^{(n)}\right),
 \end{equation}
 with infinitely many $\Gamma_{m}^{(n)}\ne \emptyset $.
 If $A$ is a regular $1$-set then it follows from Theorem \ref{g98} that $\pazocal{L}_2(A+\Gamma) \geq \pazocal{L}_2 \left(A+\Gamma_{m}^{(n)} \right) >0$ for a $\Gamma_{m}^{(n)}\ne \emptyset $.

 Assume that $A$ is an irregular set.
 Then our goal is to prove that
 \begin{equation}\label{r79}
   \pazocal{L}_2(A+\Gamma) \leq
\pazocal{L}_2\left(A+\Gamma_Z\right)+   \sum\limits_{n,m=1}^{\infty }\pazocal{L}_2\left(A+\Gamma_{m}^{(n)}\right)=0.
 \end{equation}
 This is so, because
 \begin{itemize}
   \item $\pazocal{L}_2\left(A+ \Gamma_Z\right)=0$ by Lemma \ref{r99} and
   \item $\pazocal{L}_2\left(A+ \Gamma_{m}^{(n)}\right)=0$ by Theorem
   \ref{g98} for each $n,m$.
 \end{itemize}
 This completes the proof when $Z$ is compact. 
 In the case that $Z$ is not compact, for each $n\in \N$ we may choose an open set $O_n$ so that $Z\subset O_n$ and $\pazocal{L}_1(O_n)< \frac{1}{n}$. Since $I\cap \left(O_n\right)^c$ is compact, the previous case applies to show that, if $A$ is a $1$-set, then $\pazocal{L}_2(A+ \Ga_{I\backslash O_n} ) = 0$ if and only if $A$ is irregular, where $\Ga_{(\cdot)}  = \mathbf{r}(\cdot)$. 
   Noting that $A+ \Ga_{\left(I\backslash O_n\right)}$ increases to $A+ \Ga_{\left(I\backslash Z\right)}$ implies that 
 $\pazocal{L}_2(A+ \Ga_{\left( I\backslash Z \right)} ) = 0$.  Finally, Lemma \ref{r99} implies that $\pazocal{L}_2(A+ \Ga_{Z} ) = 0$. 
\end{proof}

\subsubsection{Establishing Lemma \ref{r99}}
It remains to establish Lemma \ref{r99}.
It follows from the inner regularity of the Lebesgue measure and the dominated convergence theorem that it suffices to establish Lemma \ref{r99} for compact sets $Z$.  
Assume then that $Z$ denotes a compact set and satisfies the hypotheses of the Lemma.  
In order to prove the lemma we need the following:

\begin{fact}[$\Gamma_Z$ can be covered by finitely many arcs of arbitrarily small length]\label{r97}
  For every $\varepsilon>0$ there exists a $\delta_0$ such that every $\delta\in(0,\delta_0)$ we can find a natural number $K$, and closed sub-arcs $W_1, \dots ,W_K \subset \Gamma$ such that
  \begin{equation}\label{r96}
    \Gamma_Z \subset \bigcup\limits_{i=1}^{K}W_i,
    \quad
      \mbox{ arclength }(W_i)\le \delta,
      \quad
          \sum\limits_{i=1}^{K}\mbox{ arclength }(W_i)<\frac{\varepsilon}{2}.
  \end{equation}
\end{fact}

\begin{proof}[Proof of Fact \ref{r97} ]
Fix an arbitrary $\varepsilon>0$ and, using that $Z$ is compact, choose a finite sequence of open intervals $\left\{U_i\right\}_{i=1}^{N}$ such that
 \begin{equation}\label{r95}
    \bigcup\limits_{i=1}^{N }U_i\supset Z
      \mbox{ and }
 \sum\limits_{i=1}^{N}|U_i|<\frac{\varepsilon}{3}.
 \end{equation}
 
 Let $\Gamma_i:=\ga(U_i)$. Then since $|\ga'|\le 1$:
 \begin{equation}\label{r93}
   \sum\limits_{i=1}^{N} \text{arclength} (\Gamma_i)<\frac{\varepsilon}{3}.
 \end{equation}
Define
 $$
 \delta_0:=\min\limits_{1 \leq i \leq N}\text{arclength}(\Gamma_i).
 $$
 Moreover we define the sub-arcs $ \Gamma'_i \subset \Gamma$ as follows: the left endpoints of $\Gamma_i$ and $\Gamma'_i$ are identical and we get $\Gamma'_i$ from $\Gamma_i$ by continuing $\Gamma_i$ in $\Gamma$ over its right-end point  as long as it is possible but $\text{arclength}(\Gamma'_i) \leq \text{arclength}(\Gamma_i)+\delta_0$. Then we fix an arbitrary $\delta\in\left(0,\delta_0\right)$ and define the sequence $\left\{\Gamma_{i,j}\right\}_{j=1}^{n_i}$ of sub-arcs of $\Gamma'_i$ such that the left end points of $\Gamma_i$ and $\Gamma_{i,1}$ are identical, the right end point of $\Gamma_i$ is contained in $\Gamma_{i, n_i}$ and
 \begin{equation}\label{r92}
 \text{arclength}(\Gamma_{i,j})\le \delta,\quad
   \text{int}\Gamma_{i,j_1}\cap\text{int}\Gamma_{i,j_2}= \emptyset \mbox{ for }
   j_1\ne j_2,\quad
   \Gamma_i \subset \bigcup\limits_{j=1}^{n_i}
   \Gamma_{i,j}.
 \end{equation}
Now, for $1\le i \le N$, 
 \begin{equation}\label{r91}
   \sum\limits_{j=1}^{n_i}
   \text{arclength}(\Gamma_{i,j}) \leq \delta_0+\text{arclength}(\Gamma_i) \leq
   2\text{arclength}(\Gamma_i).
 \end{equation}
It follows from the last two displayed formulas and from \eqref{r93} that
$K:=\sum\limits_{i=1}^{N}n_i$ and $\left\{W_1, \dots ,W_K\right\}:=\bigcup\limits_{i=1}^{N}\left\{\Gamma_{i,1}, \dots ,\Gamma_{i,n_i}\right\}$  satisfy \eqref{r96}.
\end{proof}

\begin{proof}[Proof of Lemma \ref{r99}]
It is enough to prove that
\begin{equation}\label{r89}
   \mathcal{H}^2\left(A\times \Gamma_Z\right)=0
\end{equation}
since $A+\Gamma_Z$ is a Lipschitz image of $A\times \Gamma_Z$.
  Fix an arbitrary $\varepsilon>0$. We choose $\delta_0$ as in Fact \ref{r97} for this $\varepsilon$. Fix an arbitrary $\rho<\delta_0$ and let $\left\{V_i\right\}_{j=1}^{\infty }$ be a sequence of open disks on $\mathbb{R}^2$ with
  \begin{equation}\label{r90}
    A \subset \bigcup\limits_{j=1}^{\infty }V_j,\
    |V_j|<\rho\mbox{ and }\
    \sum\limits_{j=1}^{\infty }|V_j|< \mathcal{H}^1(A)+1<\infty ,
  \end{equation}
  where $| \cdot |$ denotes the diameter.
  For every $j$ let $\delta_j:=|V_j|$ and we choose sub-arcs of $\Gamma$,  $\left\{W_{1}^{j}, \dots ,W_{K_j}^{j}\right\}$ as in Fact \ref{r97} for $\delta_j$ in place of $\delta$ (so that $ \mbox{arclength}(W_i^j)\le \delta_j= |V_j| < \rho < \delta_0$).
  Then
  $$
 \left\{ \left\{V_j\times W_{i}^{j} \right\}_{i=1}^{K_j} \right\}_{j=1}^{\infty }
  $$
  is a $\sqrt{2} \cdot \rho$-cover of $A\times \Gamma_Z$ and
  $$
  \mathcal{H}_{\rho}^{2}(A\times \Gamma_Z)
   \leq
   \sum\limits_{j=1}^{\infty }
   \sum\limits_{i=1}^{K_j}|V_j\times W_{i}^{j}|^2
    \leq \sum\limits_{j=1}^{\infty }\delta_j\underbrace{\sum\limits_{i=1}^{K_j}\delta_j}_{<\frac{\varepsilon}{2}} <
    \frac{\varepsilon}{2}\sum\limits_{j=1}^{\infty }|V_j|<\frac{\varepsilon}{2} \left(\mathcal{H}^1(A)+1 \right).
  $$
  This implies that \eqref{r89} holds.
\end{proof}

\section{Proof of Theorem \ref{084}  }\label{P98}


\begin{proof}[Proof of Theorem \ref{084}]
For $N\in \N$ even, 
define the one-parameter self-similar iterated function system (IFS)
\begin{equation}\label{081}
   \mathcal{S}_\lambda:=\left\{S_{k,j,\lambda}(x):=
   \lambda x+t_{k,j}
   \right\}_{j=1,2,k=1, \dots ,N/2},\qquad
\end{equation}
where
$$
 \lambda\in U \subset \left(0,\frac{1}{N-1}\right) \qquad \text{and} \qquad 
t_{k,j}^\theta:=\mathrm{proj}_\theta(t_{k,j})\in \ell _{\theta^\bot}.
$$
For every $i=1, \dots ,n$,
the  $\mathrm{proj}_{\alpha_i}$ projection  of $\mathcal{S}_\lambda$
to the line $\ell _{\alpha_i}$ is  self-similar IFS
$$
 \mathcal{S}^{\alpha_i}_\lambda:=\left\{S^{\alpha_i}_{k,j,\lambda}(x):=
   \lambda x+t^{\alpha_i}_{k,j}
   \right\}_{j=1,2,k=1, \dots ,N/2},\qquad \lambda\in U.
$$
We write $A_\lambda$ for the attractor of the original system given in \eqref{081} and
$A_{\lambda}^{\alpha_i}$ is the attractor of the projected system. Since the linear parts of the mappings  of the original system are diagonal
(they are homothopies) 
we therefore have
\begin{equation}\label{078}
A_{\lambda}^{\alpha_i}= \mathrm{proj}_{\alpha_i}(A_\lambda).
\end{equation}
We choose the translation parameters $t_{k,j}$ such that
\begin{itemize}
  \item the Strong Separation Condition (SSC)
holds. That is for $(k_1,j_1)\ne(k_2,j_2)$ we have
$$
S_{k_1,j_1,\lambda}(A_\lambda)\cap S_{k_2,j_2,\lambda}(A_\lambda)=\emptyset.
$$
  \item $
 t^{\alpha_i}_{i,1} =t^{\alpha_i}_{i,2}
$ holds for all $i=1, \dots ,n$.
\end{itemize}
Then by the well known Hutchinson Theorem (see \cite[Section 8.3]{Falc86}) we have
\begin{equation}\label{g82}
  \dim_{\rm H}(A_\lambda)= s(\lambda):=\frac{\log N}{-\log \lambda}.
\end{equation}
Observe that for every $i=1, \dots ,n$ the projected IFS $\mathcal{S}^{\alpha_i}_\lambda$ consist of only at most  $N-1$ \underline{different} similarities with ratio  $\lambda<1/(N-1)$. This is so because
$
 t^{\alpha_i}_{i,1} =t^{\alpha_i}_{i,2}
$
and in this way
\begin{equation}\label{070}
  S_{i,1}^{\alpha_i}\equiv S_{i,2}^{\alpha_i}
\end{equation}
holds for all $i$.

The \textbf{proof of part (a')}
Choose an arbitrary $\lambda_1\in\left(\frac{1}{N},\frac{1}{N-1}\right)$.

Then by \eqref{070} and by  the choice of $\lambda_1$ for all $i=1, \dots ,n$ we have
\begin{equation}\label{075}
  \dim_{\rm H} A_{\lambda_1}>1 \mbox{ but }
  \dim_{\rm H} (A_{\lambda}^{\alpha_i})<1
\end{equation}
consequently,
\begin{equation}\label{071}
  \pazocal{L}_1\left(A_{\lambda}^{\alpha_i}\right)=0.
\end{equation}
Then by \eqref{078} and Part (b) of Lemma \ref{g94} we obtain that $\pazocal{L}_1(A+\Gamma)=0$.

 \medskip

The \textbf{proof of part (b')} of Theorem \ref{084} is analogous.  Let $\lambda_2\in\left(0,\frac{1}{N}\right)$.
 Then by \eqref{070}
 \begin{equation}\label{069}
   \dim_{\rm H} \left(A_{\lambda_2}^{\alpha_i}\right)
    \leq \frac{\log(N-1)}{-\log \lambda}<
    \frac{\log N}{-\log \lambda}
    =\dim_{\rm H} (A_{\lambda_2}).
 \end{equation}

\begin{equation}\label{068}
  A_{\lambda_2}+I_i\subset
  u_i \cdot \mathbf{e}_{\alpha_{i}^{\bot}}
  +
   A_{\lambda_2}^{\alpha_i}
   +
\ell_{\alpha_{i}^{\bot}}
\end{equation}
Hence,
$
\dim_{\rm H} \left(A_{\lambda_2}+I_i\right)
=\dim_{\rm H} \left( A_{\lambda_2}^{\alpha_i}\right)+1
<
\dim_{\rm H} (A_{\lambda_2})+1.
$
Thus,
$$
\dim_{\rm H} \left(A_{\lambda_2}+\Gamma\right)
=\max_{i=1, \dots ,n}\dim_{\rm H} \left(A_{\lambda_2}+I_i\right)
<
\dim_{\rm H} (A_{\lambda_2})+1.
$$
This completes the proof  of the second part of Theorem \ref{084}.

The \textbf{proof of part (c')} is easy. Namely, let
both $A$ and $B$ be subsets of the $x$-axis.
Let
 $A\subset[0,1]$ be a set of
$\pazocal{L}_1(A)=0$ but $\dim_{\rm H} A=1$. And let $B:=[0,1]$.
\end{proof}

\begin{acknowledgment}
We would also like to thank Alex Iosevich and Boris Solomyak for their helpful comments in improving this manuscript.  
\end{acknowledgment}


\end{document}